\documentclass[10pt]{article}

\usepackage{amsmath}
\usepackage{amsfonts}
\usepackage{amsthm}
\usepackage[english]{babel}
\usepackage{graphicx}
\usepackage[all]{xy}
\usepackage{array,calc,youngtab}
\usepackage{amssymb,stmaryrd}
\usepackage{amscd}

\newtheorem{theorem}{Theorem}[section]
\newtheorem{lemma}{Lemma}[section]
\newtheorem{definition}{Definition}[section]
\newtheorem{corollary}{Corollary}[section]
\newtheorem{proposition}{Proposition}[section]
\newtheorem{remark}{Remark}[section]
\newtheorem{example}{Example}[section]

\newtheorem*{nnthm}{Theorem}

\newcommand{\mC}{\mathbb{C}}
\newcommand{\mN}{\mathbb{N}}

\newcommand{\mZ}{\mathbb{Z}}

\newcommand{\mF}{\mathbb{F}}
\newcommand{\mV}{\mathbb{V}}
\newcommand{\mW}{\mathbb{W}}
\newcommand{\mU}{\mathbb{U}}

\newcommand{\mg}{\mathfrak{g}}
\newcommand{\ma}{\mathfrak{a}}
\newcommand{\Mp}{\mathfrak{p}}

\newcommand{\mn}{\mathfrak{n}}
\newcommand{\mb}{\mathfrak{b}}
\newcommand{\mh}{\mathfrak{h}}
\newcommand{\ml}{\mathfrak{l}}
\newcommand{\mc}{\mathfrak{c}}

\newcommand{\cC}{\mathcal{C}}

\newcommand{\cU}{\mathcal{U}}

\newcommand{\cJ}{\mathcal{J}}

\newcommand{\tto}{\twoheadrightarrow}

\newcommand{\osp}{\mathfrak{osp}(m|2n)}

\begin{document}
\title{Bernstein-Gelfand-Gelfand resolutions for basic classical Lie superalgebras}

\author{Kevin\ Coulembier\thanks{Postdoctoral Fellow of the Research Foundation - Flanders (FWO), E-mail: {\tt Coulembier@cage.ugent.be}} }

\date{\small{Department of Mathematical Analysis\\
Faculty of Engineering and Architecture -- Ghent University\\ Krijgslaan 281, 9000 Gent,
Belgium\\
\vspace{2mm}
School of Mathematics and Statistics\\ 
University of Sydney \\
Sydney, New South Wales 2006, Australia}
}

\maketitle

\begin{abstract}
We study Kostant cohomology and Bernstein-Gelfand-Gelfand resolutions for finite dimensional representations of basic classical Lie superalgebras. For each choice of parabolic subalgebra and representation of such a Lie superalgebra, there is a natural definition of the boundary and coboundary operator, which define (co)homology of the nilradical of the parabolic subalgebra. We prove that complete reducibility of the homology groups is a necessary condition to have a resolution of an irreducible module in terms of (generalised) Verma modules. Every such a resolution is then given by modules induced by these homology groups. We also prove that if these homology groups are completely reducible, a sufficient condition for the existence of this resolution is the property that these groups are isomorphic to the kernel of the Kostant quabla operator, which is equivalent with disjointness of the boundary and coboundary operator. Then we use these results to derive very explicit conditions under which BGG resolutions exist, which are particularly useful for the superalgebras of type I. For the unitarisable representations of $\mathfrak{gl}(m|n)$ and $\mathfrak{osp}(2|2n)$ we derive conditions on the parabolic subalgebra under which the BGG resolutions exist. We also apply the obtained theory to construct specific examples of BGG resolutions for $\mathfrak{osp}(m|2n)$.
\end{abstract}


\section{Introduction and main results}
\label{intro}

The strong Bernstein-Gelfand-Gelfand resolutions are resolutions of finite dimensional irreducible representations of semisimple Lie algebras in terms of direct sums of (generalised) Verma modules, see \cite{MR0578996, lepowsky}. One of the original motivations to study such resolutions is the connection with algebra (co)homology as studied in \cite{Bott, MR0142696}. The BGG resolutions and their corresponding morphisms between generalised Verma modules possess an interesting dual side as invariant differential operators, see \cite{MR1856258, Cap, MR2861216}. These differential operators have applications in many areas, see references in \cite{MR1856258, MR2861216} or e.g. \cite{MR2180410} for a concrete application. In Section \ref{classical} we give a brief historical overview of the development of the classical BGG resolutions and the corresponding differential operators, which is relevant to explain the approach taken for Lie superalgebras in the current paper. 

In \cite{Gabber} BGG resolutions for certain infinite dimensional highest weight representations of Lie algebras were studied. The BGG resolutions have also already been extended to Kac-Moody algebras in \cite{MR066169} and to some infinite dimensional Lie superalgebras in \cite{MR2646304}. They are also known to exist for certain unitary infinite dimensional representations for orthosymplectic superalgebras, see \cite{ChLW}. In \cite{MR1314151} it was studied for which quasi-hereditary algebras (with strong exact Borel subalgebras) all irreducible modules have a BGG resolution of finite length.

The first results on BGG resolutions for finite dimensional modules of Lie superalgebras were obtained in \cite{Cheng}. It was proved that these do not exist in full generality for basic classical Lie superalgebras, but they were obtained for the tensor modules of $\mg=\mathfrak{gl}(m|n)$ and for the parabolic subalgebra such that the Levi subalgebra is equal to the underlying Lie algebra $\mg_{\overline{0}}=\mathfrak{gl}(m)\oplus\mathfrak{gl}(n)$. These resolutions in \cite{Cheng}, which are in terms of Kac modules, were extended in \cite{MR2600694} from tensor modules to so-called Kostant modules. The results were also extended in \cite{MR2670923} by a powerful equivalence of categories between subcategories of the parabolic BGG categories of $\mathfrak{gl}(m+\infty)$ and $\mathfrak{gl}(m|\infty)$, known as super duality. This led to the result that the tensor modules of $\mathfrak{gl}(m|n)$ have BGG resolutions for each parabolic subalgebra which contains $\mathfrak{gl}(n)$. In \cite{osp12n} it was proved that every finite dimensional module of $\mathfrak{osp}(1|2n)$ can be resolved in terms of Verma modules. This result can be extended to typical blocks for basic classical Lie superalgebras. Another way to obtain the BGG resolutions for typical simple modules is through the Morita equivalence in \cite{MR1862800}.

In this paper $\mg=\mg_{\overline{0}}\,+\,\mg_{\overline{1}}$ will always stand for a basic classical Lie superalgebra. The classification of basic classical Lie superalgebras restricts to $\mathfrak{sl}(m|n)$ (for $m\not=n$), $\mathfrak{psl}(n|n)$, $\osp$, $F(4)$, $G(3)$ and $D(2,1;\alpha)$, see \cite{MR0486011}. The obtained results also hold for $\mathfrak{gl}(m|n)$ and for convenience we will not always make an explicit distinction between $\mathfrak{gl}(m|n)$ and $\mathfrak{sl}(m|n)$. The notation $\Mp$ is used for a parabolic subalgebra, i.e. a subalgebra containing a Borel subalgebra of $\mg$. The maximal contragedient subalgebra of $\Mp$, also known as the Levi subalgebra of $\Mp$, is denoted by $\ml$. The parabolic subalgebra then has a vector space decomposition $\Mp=\ml+\mn$ and $\mg$ decomposes likewise as $\mg=\overline{\mn}+\ml+\mn$. We also use the notation $\Mp^\ast=\overline{\mn}+\ml$. The symbols $\mV$ and $\mW$ will be used for finite dimensional irreducible $\mg$-modules. If we want to mention explicitly the highest weight $\lambda$ of the representation, it is denoted by $\mV_\lambda$ or $\mW_\lambda$. The notation $[\cdot,\cdot]_\ma$ will always stand for the projection of the Lie superbracket onto a subspace $\ma\subset\mg$ with respect to a naturally defined complement space $\mc$, $\mg=\ma+\mc$.

We take a new systematic, direct approach to construct BGG resolutions for basic classical Lie superalgebras, based on considerations about Kostant (co)homology. This corresponds to (co)homology of the nilradical $\mn$ of the parabolic subalgebra and its dual in irreducible representations of $\mg$. The invariant bilinear form (usually the Killing form) on these algebras allows to connect homology in $\mn$ and cohomology in $\overline{\mn}$ as in \cite{MR0142696}. As in the classical case there is an intimate relation between the Kostant cohomology and BGG resolutions. We prove the fact that a proper BGG resolution yields the homology groups, which is a well-known homological property for Lie algebras. But we also prove how, under certain assumptions, BGG resolutions can be obtained starting from the properties of the cohomology.

 The first result we obtain on BGG resolutions for Lie superalgebras is the following consequence of Theorem \ref{onlyHk} and Corollary \ref{onlyHk2}. This theorem follows from homological arguments and the fact that a resolution in terms of generalised Verma modules is a projective resolution in the category of $\overline{\mn}$-modules.
\begin{nnthm}
Consider a basic classical Lie superalgebra $\mg$ with parabolic subalgebra $\Mp=\ml+\mn$. If the irreducible finite dimensional $\mg$-module $\mW$ has a resolution in terms of direct sums of generalised Verma modules, the homology groups $H_k(\overline{\mn},\mW)$ are completely reducible as $\ml$-modules and there is a resolution of the form
\begin{eqnarray}
\label{onlyform}
\cdots \to V^{H_k(\overline{\mn},\mW)}\to\cdots\to V^{H_1(\overline{\mn},\mW)}\to V^{H_0(\overline{\mn},\mW)}\to\mW\to0
\end{eqnarray}
with $V^{H_k(\overline{\mn},\mW)}=\cU(\mg)\otimes_{\cU(\Mp)}H_k(\overline{\mn},\mW)$. This resolution does not decompose into a non-trivial direct sum of complexes and every resolution of $\mW$ in terms of generalised Verma modules, which does not decompose, is of the form above.
\end{nnthm}

Therefore the next aim is to derive conditions under which a resolution of the form \eqref{onlyform} exists; knowing that if the homology groups are completely reducible, this is the only proper generalisation of the BGG resolutions and if they are not completely reducible, a proper BGG resolutions does not exist. The next central result we obtain is that if the boundary and coboundary operator on $C_\bullet(\overline{\mn},\mV)=C^\bullet({\mn},\mV)$ are disjoint (as always holds in the Lie algebra case), then a resolution in terms of modules induced by the (not necessarily completely reducible) homology groups $H_{\bullet}(\overline{\mn},\mV)$ exists, in Theorem \ref{finalthm}. Therefore we derive practical criteria for this disjointness in Theorem \ref{Boxcohom2}. These results lead to the following theorem.
\begin{nnthm}
Consider $\mg$, $\Mp=\ml+\mn$ and $\mW$ as in the previous theorem. The image of the boundary operator on $C_\bullet(\overline{\mn},\mV)$ has trivial intersection with the kernel of the Kostant quabla operator $\Box$ if and only if the homology groups $H_{\bullet}(\overline{\mn},\mW)$ are isomorphic to the generalised zero eigenspace of the quabla operator on $C_\bullet(\overline{\mn},\mV)$. If one of these conditions holds, there exists a resolution of $\mW$ as in equation \eqref{onlyform}.

Consequently, if $H_{\bullet}(\overline{\mn},\mW)$ is completely reducible and isomorphic to $\ker\Box$, there exists a BGG resolution for $\mW$.
\end{nnthm}
The way we obtain this result is through the dual side in terms of coinduced modules, where inspiration is drawn from the approach to curved parabolic geometries in \cite{Cap}. The coinduced module can be identified with the module inducing the infinite jet prolongation of the homogeneous vector bundle $G\times_p\mV\to G/P$. This approach also has the advantage that the resolutions we obtain are ready to be expressed in terms of their dual side as differential operators, which can be of importance regardless of complete reducibility of the homology groups. We also provide a simple counterexample to show that the obtained sufficient condition $H_{\bullet}(\overline{\mn},\mW)=\ker\Box$ (with $H_{\bullet}(\overline{\mn},\mW)$ completely reducible), is not a necessary condition for a module to be resolved by generalised Verma modules, in Example \ref{counter}.

Our construction of BGG resolutions is particularly applicable to basic classical Lie superalgebras of type $I$, which possess large classes of unitarisable representations, also called star representations, see \cite{MR1075732, MR0424886, MR1067631}. This leads to the following results in Theorem \ref{mainthmgl} and Theorem \ref{mainthmC}.
\begin{nnthm}
For $\mg=\mathfrak{gl}(m|n)$ with a parabolic subalgebra $\Mp$, the $\mathfrak{g}$-module $\mV$ can be resolved in terms of direct sums of $\Mp$-Verma modules if
\begin{itemize}
\item $\mV$ is a star representation of type $(1)$ and $\Mp$ contains $\mathfrak{gl}(n)$
\item $\mV$ is a star representation of type $(2)$ and $\Mp$ contains $\mathfrak{gl}(m)$.
\end{itemize}
For $\mg=\mathfrak{osp}(2|2n)$ with a parabolic subalgebra $\Mp$, the $\mathfrak{g}$-module $\mV$ can be resolved in terms of direct sums of $\Mp$-Verma modules if
\begin{itemize}
\item $\mV$ is a star representation of type $(1)$ and $\Mp$ contains $\mathfrak{sp}(2n)$
\item $\mV$ is a star representation of type $(2)$.
\end{itemize}
\end{nnthm}
Since the tensor modules of $\mathfrak{gl}(m|n)$ are included in the class of star representations of type $(1)$, see \cite{MR1075732}, this extends the results on BGG resolutions in \cite{Cheng} and \cite{MR2670923}, the resolutions for $\mathfrak{osp}(2|2n)$ are new. The star representations of type $(2)$ are in both cases the dual representations of those of type $(1)$.

Also non-unitarisable cases can be studied. From our general construction of resolutions, the following result in Theorem \ref{criterion} and Theorem \ref{Boxcohom2} can be derived.
\begin{nnthm}
Consider a basic classical Lie superalgebra $\mg$, with parabolic subalgebra $\Mp=\ml+\mn$ and a finite dimensional irreducible representation $\mW$. Assume the $\ml$-module $\ker\Box\subset C^\bullet(\mn,\mW)$ is completely reducible. If for each $k\in\mN$ it holds that
\begin{eqnarray*}
[\ker\Box_k:M][\ker\Box_{k+1}:M]&\le &1
\end{eqnarray*}
for all irreducible $\ml$-modules $M$, then $\mW$ has a resolution in terms of direct sums of generalised Verma modules. Moreover, in this case $H_{\bullet}(\overline{\mn},\mW)\cong H^\bullet(\mn,\mW)\cong \ker \Box$ holds.
\end{nnthm}

These results are applied to construct BGG resolutions motivated by the study in \cite{CSS} of the super Laplace operator as an intertwining operator between principal series representations for $\osp$. This operator and its symmetries appear in certain quantum mechanical problems in superspace, see \cite{MR2852289, MR2395482}. In \cite{CSS} the quotient of the universal enveloping algebra of $\osp$ with respect to a Joseph-like ideal was identified as an algebra of symmetries of the super Laplace operator. In the classical case for $\mathfrak{so}(m)$, the completeness of this algebra of symmetries follows from certain BGG resolutions, see \cite{MR2180410}. In the current paper we approach this problem for $\osp$. Therefore we take the Lie superalgebra $\mathfrak{osp}(m|2n)$ with maximal parabolic subalgebra as in \cite{MR2395482}. In particular it follows that the natural representation of $\osp$ can be resolved in terms of generalised Verma modules if the condition $m-2n\le 1$ holds.

We also state some results on BGG resolutions for Kac modules of basic classical Lie superalgebras of
type I, which are parabolically induced modules. This rederives the existence of certain BGG resolutions
for typical highest weight modules. Finally we note how singly atypical modules, see \cite{MR1063989}, can be resolved
by Kac modules.

In the current paper we do not focus on the weak BGG resolutions. Often the strong resolutions are derived from the weak resolutions, see \cite{MR0578996, Cheng, osp12n, MR2428237, lepowsky, MR066169}, but in the current paper we take a different approach.

The remainder of the paper is organised as follows. In Section \ref{classical} we give a brief review of the BGG resolutions, Kostant cohomology and BGG sequences of invariant differential operators on parabolic geometries, for Lie algebras. In Section \ref{KillingStar} we derive some results on non-degenerate invariant bilinear forms and adjoint operations on basic classical Lie superalgebras and the corresponding contravariant hermitian forms on representations. In Section \ref{Kostant} we introduce the (co)boundary operator on $C_\bullet(\overline{\mn},\mV)$ and $C_\bullet({\mn},\mW)$. Then we derive some properties, which are well-known for Lie algebras and work out some formulae that will be useful for the sequel. The main difference with the classical case is the fact that the boundary and coboundary operator are not necessarily disjoint operators. Some results on how to handle this new feature are derived, in particular we obtain several criteria for this disjointness. In Section \ref{BGG} we prove that any proper BGG resolution is given in terms of modules induced by the homology groups, which implies that the homology groups need to be completely reducible in order to allow BGG resolutions. Then we introduce an invariant operator $d$ acting between the $\mg$-modules coinduced by the $\Mp$-modules $C_\bullet({\mn},\mV)$. This leads to a coresolution of $\mV$. If the boundary and coboundary operator are disjoint we can introduce a splitting operator, which is applied to construct a coresolution of $\mV$ in terms of $\mg$-modules coinduced by the homology groups $H_\bullet(\mn,\mV)$. By dualising this result we obtain the desired resolution of $\mV^\ast$ in terms of induced modules in Theorem \ref{finalthm}, generalising the BGG resolutions for Lie algebras. In Section \ref{ExamGL} we focus on the unitarisable representations of $\mathfrak{gl}(m|n)$ and $\mathfrak{osp}(2|2n)$, which leads to extensive classes of cases where BGG resolutions exist. In Section \ref{ExamOSp} we derive alternative and practical sufficient conditions for the disjointness of the boundary and coboundary operator, that do not require unitarisability or even complete reducibility of $C_\bullet(\overline{\mn},\mV)$. Using this, we focus on $\osp$ and a maximal parabolic subalgebra. Then we briefly pay attention to typical highest weight modules of basic classical Lie superalgebras of type $I$. Finally, in the Appendix we study the operator $d$ using Hopf algebraic techniques and show that it corresponds to a twisted exterior derivative. This result is needed to prove the exactness of the operator $d$ in Section \ref{BGG}.

\section{Lepowsky BGG resolutions and invariant differential operators}
\label{classical}
In \cite{MR0578996}, Bernstein, Gel'fand and Gel'fand proved that each finite dimensional irreducible representation $\mV_\lambda$ of a complex semisimple Lie algebra $\mg$ has a resolution in terms of Verma modules. In \cite{lepowsky} Lepowsky extended this result to the parabolic setting. To state his result we need to introduce the notation $M(\lambda)$ for the irreducible $\Mp$-representation with highest weight $\lambda\in\mathfrak{h}^\ast$ with $\mathfrak{h}$ the Cartan subalgebra of $\mg=\overline{\mn}+\ml+\mn$. For the irreducible $\mg$-representation $\mV_\lambda$, there is an exact complex 
\begin{eqnarray}
\label{resLep}
0\rightarrow\bigoplus_{w\in W^1(\dim\mn)}V^{M(w\cdot\lambda)} \to\cdots\to \bigoplus_{w\in W^1(j)}V^{M(w\cdot\lambda)}\to\cdots\to \bigoplus_{w\in W^1(1)}V^{M(w\cdot\lambda)}\to V^{M(\lambda)} \to \mV_\lambda\to0
\end{eqnarray}
with $W^1$ a subset of the Weyl group corresponding to the quotient of the Weyl group with the Weyl group of the Levi algebra (with $\rho$-shifted action on $\mh^\ast$) and $V^M=\cU(\mg)\otimes_{\cU(\Mp)}M$ the parabolic (or generalised) Verma module induced by the irreducible $\Mp$-module $M$, see \cite{lepowsky} for details.

Since such a resolution is a projective resolution in the category of finitely generated $\overline{\mn}$-modules, this allows to compute the homology of $\overline{\mn}$ in the module $\mV_\lambda$. This homology was  already studied by Kostant in \cite{MR0142696}, the homology groups satisfy
\begin{eqnarray*}
H_{j}(\overline{\mn},\mV_\lambda)&\cong& \bigoplus_{w\in W^1(j)}M(w\cdot\lambda).
\end{eqnarray*}
The result of Lepowsky can therefore be rewritten in terms of these homology groups.

There is a well-known correspondence between generalised Verma module morphisms and differential operators acting between the principal series representations, corresponding to vector bundles on the generalised flag manifolds $G/P$ with $G$ and $P$ groups with Lie algebras $\mg$ and $\Mp$, see e.g. \cite{Dobrev, MR0430166}. This implies that there is a locally exact complex
\begin{eqnarray*}
&&0\quad\to\quad \mV^\ast\quad\to\quad\Gamma(G/P,G\times_PH^0(\overline{\mn},\mV^\ast))\quad\to\quad\Gamma(G/P,G\times_PH^1(\overline{\mn},\mV^\ast))\quad\to\,\cdots\\
&&\cdots\,\to\quad\Gamma(G/P,G\times_PH^{\dim \mn}(\overline{\mn},\mV^\ast))\quad\to \quad0.
\end{eqnarray*}
One of the interesting features of this result is that each irreducible representation $\mV^\ast$ can be explicitly realised as the kernel of some set of differential operators
\[\Gamma(G/P,G\times_PH^0(\overline{\mn},\mV^\ast))\quad\to\quad\Gamma(G/P,G\times_PH^1(\overline{\mn},\mV^\ast)).\] 
Since this only depends on the last part of the BGG resolution (or the first part of its dual differential operator sequence) we will pay special attention to this last part in Section \ref{ExamOSp}. There we find some cases in which the BGG resolutions might not exist, but this last part still has an analogue. This realisation of $\mV^\ast$ in $\Gamma(G/P,G\times_PH^0(\overline{\mn},\mV^\ast))$ is similar to the Borel-Weil theorem. The Borel-Weil(-Bott) theorem of \cite{Bott} has been studied for Lie supergroups in \cite{MR2734963, MR0957752, MR1036335, MR2059616}.

In \cite{Cap} \v{C}ap, Slov\'ak and Sou{\v{c}}ek proved that the differential operators in the sequence above extend to curved Cartan geometries based on $G/P$, known as parabolic geometries, even though there the sequences are no longer complexes. In doing so, they also provided a new proof of the BGG differential operators for the flat model $G/P$, which gives a new proof of the result of Lepowsky. It turns out that the ideas in \cite{Cap}, using differential operators, extends more easily to the supersetting than the direct proof of the BGG resolutions in \cite{lepowsky}. In this paper we use methods inspired by the simplification of \cite{Cap} provided in \cite{MR1856258} by Calderbank and Diemer, to prove BGG resolutions for basic classical Lie superalgebras. Even though some of the machinery is inspired by the differential operator side we will formulate and prove everything in a purely algebraic setting in the current paper.

\section{Killing form and adjoint operations}
\label{KillingStar}

In this paper we will use two types of forms. One type takes the $\mZ_2$-gradations into account and the other ignores that structure. It will always hold that the relevant notions for bilinear forms are supersymmetry and graded invariance. The relevant notions for sesquilinear forms are hermicity (not superhermicity) and contravariance. This different behaviour for the two types of forms is also reflected in the extension to tensor products.

We recall the notions of adjoint operation and star representation of a complex Lie superalgebra, see \cite{MR0424886}. A star Lie superalgebra is equipped with a map $\dagger:\mg\to\mg$ which is antilinear, even and satisfies $[A,B]^\dagger=[B^\dagger,A^\dagger]$ and $\left(A^\dagger\right)^\dagger=A$ for $A,B\in\mg$. Such a map is called an adjoint operation. A star representation (or unitarisable representation) of such an algebra is a representation with a (positive definite, hermitian) inner product $\langle\cdot,\cdot,\rangle$ on $\mV$ which satisfies $\langle Av,w\rangle=\langle v,A^\dagger w\rangle$ for $v,w\in\mV$ and $A\in\mg$. An inner product satisfying this last property is called contravariant.

For every adjoint operation on a basic classical Lie superalgebra, the Cartan element corresponding to a certain simple root can be considered to be mapped onto itself under the adjoint operation. Other adjoint operations correspond to combinations of such adjoint representations with Lie superalgebra isomorphisms.
\begin{lemma}
\label{existadj}
Every basic classical Lie superalgebra has an adjoint operation.
\end{lemma}
\begin{proof}
For all basic classical Lie superalgebras except $D(2,1;\alpha)$ this follows from \cite{MR0424886}. For the Lie superalgebra $D(2,1;\alpha)$ this follows from \cite{MR0787332}.
\end{proof}

Such an adjoint operation can always be used to construct a non-degenerate contravariant hermitian form on a finite dimensional irreducible representation of a basic classical Lie superalgebra, this is known as the Shapovalov form. For completeness we prove this fact in the following lemma. In case this form is positive definite, the representation is a star representation. For finite dimensional basic classical Lie superalgebras, this positive definiteness is only possible for those of type I, $A(m|n)=\mathfrak{sl}(m|n)$ and $C(n)=\mathfrak{osp}(2|2n)$, see \cite{MR0424886, MR0787332}.
\begin{lemma}
\label{contraform}
Every simple highest weight module $\mV$ of a basic classical Lie superalgebra $\mg$, has a non-degenerate contravariant hermitian form. This is a non-degenerate hermitian form $\langle\cdot,\cdot,\rangle$ satisfying
\begin{eqnarray*}
\langle A v,w\rangle&=&\langle v,A^\dagger w\rangle,
\end{eqnarray*}
for $A\in\mg$, $v,w\in\mV$ and $\dagger$ an adjoint operation on $\mg$.
\end{lemma}
\begin{proof}
First we define a contravariant hermitian form on the corresponding generalised Verma module. If the highest weight vector of $\mV$ is of weight $\lambda$, we consider the generalised Verma module $V^\lambda=\cU(\mg)\otimes_{\cU(\mb)}\mC_\lambda$. The highest weight vector of the generalised Verma module will be denoted by $\widetilde{v}_+$.

Consider an adjoint operation, which exists by Lemma \ref{existadj}. We define the form $\langle\cdot,\cdot\rangle$ on $V^\lambda$ by putting $\langle \widetilde{v}_+,\widetilde{v}_+\rangle=1$ and $\langle \widetilde{v}_+,w\rangle=0$ for every weight vector $w\in V^\lambda$ of lower weight and
\begin{eqnarray*}
\langle Y_1\cdots Y_k\widetilde{v}_+,Y_1'\cdots Y_l'\widetilde{v}_+\rangle=\langle \widetilde{v}_+,Y_k^\dagger\cdots Y_1^\dagger Y_1'\cdots Y_l'\widetilde{v}_+\rangle
\end{eqnarray*} 
for $Y_1,\cdots, Y_k$ and $Y_1',\cdots,Y_l'$ negative root vectors. This is consistently defined by the properties of an adjoint operation. The adjoint operation can be naturally extended to $\cU(\mg)$, by the relation $(A_1\cdots A_p)^\dagger=A_p^\dagger\cdots A_1^\dagger$ for $A_1,\cdots,A_p\in\mg$. From this approach and the knowledge of the action of $\dagger$ on the Cartan subalgebra, it follows easily that the proposed form on $V^\lambda$ is hermitian. 

Clearly the maximal submodule of the generalised Verma module consists of vectors which are orthogonal on all vectors. Therefore the form descends to $\mV$. Since $\mV$ is simple and every degenerate subspace is an ideal, this bilinear form has to be non-degenerate.
\end{proof}

By definition, every basic classical Lie superalgebra $\mg$ admits a non-degenerate supersymmetric consistent invariant bilinear form, denoted by $(\cdot,\cdot):\mg\times\mg\to\mC$. This form is unique up to a normalisation. Invariant means that
\begin{eqnarray*}
([B,A],C)=-(-1)^{|A||B|}(A,[B,C])&\mbox{or equivalently}&([A,B],C)=(A,[B,C])\quad\mbox{holds}\quad\mbox{for }A,B,C\in\mg.
\end{eqnarray*}
Consistent means $(\mg_{\overline{0}},\mg_{\overline{1}})=0$ and supersymmetric means $(A,B)=(-1)^{|A||B|}(B,A)$.

For most basic classical Lie superalgebras this form is given by the Killing form, see \cite{MR0486011}. Every other invariant form is proportional to the Killing form. For $\mathfrak{osp}(2n+2|2n)$ and $D(2,1;\alpha)$ the Killing form is identically zero but another such form can be constructed. For simplicity we refer to this form as the Killing form for all basic classical Lie superalgebras. In the following proposition we prove the well-known fact that this Killing form always exists in way that will be useful for the sequel.
\begin{proposition}
\label{Killingfromstar}
\label{Killingexist}
There exists a non-degenerate supersymmetric consistent invariant bilinear form on every basic classical Lie superalgebra. 
\end{proposition}
\begin{proof}
Since $\mg$ is simple, the adjoint representation has a non-degenerate contravariant hermitian form $\langle\cdot,\cdot\rangle$ as in Lemma \ref{contraform} for an adjoint operation $\dagger$. The bilinear forms $(\cdot,\cdot)_1$ and $(\cdot,\cdot)_2$ defined by
\begin{eqnarray*}
(A,B)_1=\langle A^\dagger, B\rangle&\mbox{and}&(A,B)_2=(-1)^{|A||B|} \langle B^\dagger, A\rangle=(-1)^{|A||B|}(B,A)_1
\end{eqnarray*}
are non-degenerate, consistent and invariant. We can construct a supersymmetric one by adding these two together, $(\cdot,\cdot)=(\cdot,\cdot)_1+(\cdot,\cdot)_2$. This is still consistent and invariant.

Since $\mg$ is simple an invariant bilinear form is either zero or non-degenerate, because a degenerate subspace would be an ideal. In order for $(\cdot,\cdot)$ to be zero, $(\cdot,\cdot)_1$ needs to be super skew symmetric. The restriction of $(\cdot,\cdot)_1$ to $\mg_{\overline{0}}$ is proportional to the Killing form of this reductive Lie algebra (and non-zero) and therefore symmetric, so $(\cdot,\cdot)_1$ can not be super skew symmetric and $(\cdot,\cdot)$ is non-degenerate.
\end{proof}

The quadratic Casimir operator is given by
\begin{eqnarray*}
\cC_2&=&\sum_{i=1}^{\dim\mg}A_i A_i^\ddagger\quad\in\cU(\mg)
\end{eqnarray*}
for $\{A_i\}$ a basis of $\mg$ and $\{A^\ddagger_i\}$ its dual basis, $(A_i^\ddagger,A_j)=\delta_{ij}$. This operator is central in $\cU(\mg)$, which is a direct consequence of the invariance of the Killing form.

In the following lemma we prove that an arbitrary adjoint operation on a basic classical Lie superalgebra has a special relation with the Killing form. In fact, it is known that there are only two adjoint operations on each basic classical Lie superalgebra, see \cite{MR0424886}, one can be derived from the other as in the subsequent Proposition \ref{stardual}.
\begin{lemma}
\label{starwithKilling}
For an adjoint operation $\dagger$ on a basic classical Lie superalgebra $\mg$, it holds that
\begin{eqnarray*}
(A^\dagger,B^\dagger)&=&\overline{(B,A)}
\end{eqnarray*}
with $A,B\in\mg$ and $(\cdot,\cdot)$ the Killing form on $\mg$ of Proposition \ref{Killingexist}.
\end{lemma}
\begin{proof}
The property follows immediately from the fact that the form can be defined as in the proof of Proposition \ref{Killingfromstar}.
\end{proof}

The dual representation of a $\mg$-representation $\mV$ is defined on the space of linear functionals $\mV^\ast$. This space becomes a $\mg$-module with action given by $(A\alpha)(v)=-(-1)^{|A||\alpha|}\alpha(Av)$ for $A\in\mg$, $\alpha\in\mV^\ast$ and $v\in\mV$. The following anti-linear bijection between $\mV$ and $\mV^\ast$ will be of importance in Section \ref{Kostant}. This lemma is equivalent to the statement that an irreducible representation is isomorphic to its own twisted dual.

\begin{lemma}
\label{bijection}
Consider a finite dimensional irreducible representation $\mV$ of a basic classical Lie superalgebra $\mg$, with dual representation $\mV^\ast$. There is an anti-linear bijection, $\psi_0:\mV\to\mV^\ast$ that satisfies $\psi_0(Av)=-(-1)^{|A||v|}A^\dagger\psi_0(v)$ for $A\in\mg$, $v\in\mV$ and $\dagger$ an adjoint operation on $\mg$.
\end{lemma}
\begin{proof}
This can be deduced from the hermitian form in Lemma \ref{contraform}. For $v\in\mV$ we define $\psi_0(v)\in\mV^\ast$ as $\left(\psi_0(v)\right)(w)=\langle v,w\rangle$ for all $w\in\mV$. This is a bijection because of the non-degeneracy of $\langle \cdot,\cdot\rangle$. Then it follows that
\begin{eqnarray*}
\left(\psi_0(Av)\right)(w)&=&\langle Av,w\rangle=\langle v,A^\dagger w\rangle\\
&=&\left(\psi_0(v)\right)(A^\dagger w)=-(-1)^{|A||v|}\left(A^\dagger\psi_0(v)\right)(w),
\end{eqnarray*}
which proves the lemma.
\end{proof}

The dual of a star representation is again a star representation, but with different adjoint operation, as is stated in the following proposition. This generalizes Proposition 1 in \cite{MR1067631}.

\begin{proposition}
\label{stardual}
If the irreducible $\mg$-module $\mV$ is a star representation for adjoint operation $\dagger$, then $\mV^\ast$ is also a star representation for adjoint operation
\begin{eqnarray*}
A&\to& (-1)^{|A|}A^\dagger\qquad\mbox{ for }A\in\mg.
\end{eqnarray*}
\end{proposition}
\begin{proof}
The inner product on $\mV^\ast$ is defined as $\langle \alpha,\beta\rangle=\sum_i\overline{\alpha(v_i)}\beta(v_i)$ for $\{v_i\}$ an orthonormal basis of $\mV$. The fact that $\langle A\alpha,\beta\rangle=(-1)^{|A|}\langle \alpha,A^\dagger \beta\rangle$ follows from a direct calculation. Also the fact that the proposed mapping yields an adjoint operation follows in a straightforward manner.
\end{proof}

Finally we define the twisted dual representation of a finite dimensional representation.
\begin{definition}{\rm
\label{twisteddual}
For any finite dimensional representation $\mU$ of a star Lie superalgebra with adjoint operation $\dagger$ the twisted dual representation $\mU^\vee$ with respect to $\dagger$ is defined on $\mU^\ast$, the space of linear functionals on $\mU$, with action given by
\[(A \alpha) (x)=\alpha(A^\dagger\cdot x) \]
for $\alpha\in \mU^\ast$, $x\in\mU$ and $A$ an element of the Lie superalgebra.}
\end{definition}
Note that an irreducible finite dimensional representation is its own twisted dual, which follows from the fact that a finite dimensional irreducible representation is completely determined by its highest weight and the parity of the corresponding vector.

\section{Kostant cohomology}
\label{Kostant}

The type of algebra (co)homology we consider was originally studied by Kostant for Lie algebras in \cite{MR0142696} and by Bott for the case where the parabolic subalgebra is a Borel subalgebra in \cite{Bott}. For Lie superalgebras, results have been obtained in e.g. \cite{Cheng, MR2646304, MR2037723, MR2036954, MR2746029}. However, all these results are for the case of unitarisable representations, while in this section we do not impose this very restrictive condition. The closely related sheaf cohomology on the generalised flag supermanifold $G/P$ (as a generalisation of \cite{Bott}) has been studied in \cite{MR2734963, MR1036335, MR2059616}.

We consider a basic classical Lie superalgebra $\mg=\overline{\mn}+\ml+\mn$ as in the introduction, with parabolic subalgebra $\Mp=\ml+\Mp$ and a finite dimensional irreducible representation $\mV$. We also use the notation $|\cdot|\in\mZ_2$ which maps a homogeneous element of a super vector space to $\overline{0}$ or $\overline{1}$ depending on whether it is even or odd. The summation $\sum_a$ will always be used to denote a summation $\sum_{a=1}^{\dim\mn}$ related to a basis $\{\xi_a\}$ of $\mn$. The notation $A\cdot v$ stands for the action of $A\in\mg$ on $v$ an element of a $\mg$-module $\mV$.

Usually we will consider bases $\{A_i\}$ of $\mg$ which split into a basis $\{\xi_a\}$ of $\mn$ and a basis $\{T_\kappa\}$ of $\Mp^\ast$.

\subsection{Definitions and basic properties for $C_\bullet({\mn},\mV)$}

Homology for $\mn$ and cohomology for $\overline{\mn}$ is defined by introducing the space of $k$-chains $C_k(\mn,\mV)=\Lambda^k\mn\otimes\mV$ and the space of $k$-cochains $C^k(\overline{\mn},\mV)=\mbox{Hom}(\Lambda^k\overline{\mn},\mV)$. Through the Killing form $(\cdot,\cdot)$ from Proposition \ref{Killingexist} we have the identification 
\begin{eqnarray*}
C_k(\mn,\mV)=\Lambda^k\mn\otimes\mV\cong\mbox{Hom}(\Lambda^k\overline{\mn},\mV)=C^k(\overline{\mn},\mV).
\end{eqnarray*}

The explicit evaluation of $\Lambda^k\mn\otimes\mV$ on $\Lambda^k\overline{\mn}$ in the identification above is inherited from the evaluation of $\otimes^k\mn\,\otimes\,\mV$ on $\otimes^k\overline{\mn}$ defined by induction through
\begin{eqnarray*}
\left(X_1\otimes \cdots\otimes X_k\otimes v\right)\left(Y_1\otimes\cdots\otimes Y_k\right)=(-1)^{|X_1|(|Y_2|+\cdots+|Y_k|)}(Y_1,X_1)\left(X_2\otimes \cdots\otimes X_k\otimes v\right)\left(Y_2\otimes\cdots\otimes Y_k\right).
\end{eqnarray*}

The spaces $C_k(\mn,\mV)$ are naturally $\Mp$-modules. We also introduce the notation $[A,V]$ and $[V,A]$ for the left and right adjoint action of $A\in\mg$ on $V\in\otimes^k\mg$, induced from the Lie superbracket which captures the adjoint action of $\mg$ on itself.

The coboundary operator $\partial: C^{\bullet}(\overline{\mn},\mV)\to C^{\bullet}(\overline{\mn},\mV)$ linearly maps elements from $C^k(\overline{\mn},\mV)$ to $C^{k+1}(\overline{\mn},\mV)$ for each $k$. Sometimes we will use $\partial_k$ to denote we consider the restriction of $\partial$ to $C^k(\overline{\mn},\mV)$ and similarly for the boundary operator $\partial^\ast$.
\begin{definition}{\rm
\label{defder}
The coboundary operator is given by $\partial_k: C^k(\overline{\mn},\mV)\to C^{k+1}(\overline{\mn},\mV)$
\begin{eqnarray*}
\left(\partial_k \alpha\right) (Y_0\wedge Y_1\wedge\cdots Y_k)&=&\sum_{i=0}^k(-1)^{i}(-1)^{|Y_i|(|Y_{0}|+\cdots +|Y_{i-1}|)}Y_i\cdot  \left(\alpha (Y_0\wedge\cdots^{\hat{i}}\cdot Y_k)\right)\\
&&+\frac{1}{2}\sum_{i=0}^k(-1)^i(-1)^{|Y_i|(|Y_{i+1}|+\cdots |Y_{k}|)} \alpha([Y_0\wedge\cdots^{\hat{i}}\cdot Y_k,Y_i])
\end{eqnarray*}
for {\rm $\alpha\in\mbox{Hom}(\Lambda^k\overline{\mn},\mV)$.} The boundary operator $\partial^\ast_k: C_k({\mn},\mV)\to C_{k-1}({\mn},\mV)$ is given by
\begin{equation}
\label{defcodiff}
 \partial^\ast_k(X\wedge f)=-X\cdot f- X\wedge \partial_{k-1}^\ast (f),
\end{equation}
for $f\in C^{k-1}(\overline{\mn},\mV)$, where we set $\partial_0^\ast=0$.}
\end{definition}

Usually the definition of $\partial$ is given by a different but equivalent expression, see \cite{MR0142696}. Here we choose a different form that leads to less arguments in the sign $(-1)$ originating from switching elements of $\mn$.

\begin{lemma}
\label{coder1}
The boundary operator $\partial^\ast: C_{\bullet}({\mn},\mV)\to C_{\bullet}({\mn},\mV)$ satisfies $\partial^\ast\circ\partial^\ast=0$ and is a $\Mp$-module morphism.
\end{lemma}
\begin{proof}
These properties follow by induction, using the definition of $\partial^\ast$ in equation \eqref{defcodiff}.
\end{proof}

Contrary to $\partial^\ast$, $\partial$ is not a $\Mp$-module morphism. In order to construct BGG resolutions we need the following lemma, which describes how far $\partial$ is from being a $\Mp$-module morphism.
\begin{lemma}
\label{pactionder}
For $Z\in\Mp$ and $f\in C^k(\overline{\mn},\mV)$ the following relation holds:
\begin{eqnarray*}
\partial (Z \cdot f)&=&Z\cdot \partial (f)+\sum_a \xi_a\wedge[\xi_a^\ddagger,Z]_{\Mp}\cdot f,
\end{eqnarray*}
for any basis $\{\xi_a\}$ of $\mn$ with dual basis $\{\xi_a^\ddagger\}$ of $\overline{\mn}$ with respect to the Killing form. In particular, $\partial$ is an $\ml$-module morphism. The coboundary operator also satisfies $\partial\circ\partial=0$.
\end{lemma}
\begin{proof}
Through the identification $\Lambda^k\mn\otimes\mV=\mbox{Hom}(\Lambda^k\overline{\mn},\mV)$ the $\Mp$-action on $\alpha\in \mbox{Hom}(\Lambda^k\overline{\mn},\mV)$ is given by
\begin{eqnarray*}
(Z\cdot \alpha)(Y_1\wedge\cdots \wedge Y_k)&=&(-1)^{|Z|(|Y_1|+\cdots+|Y_k|)}Z\cdot\left(\alpha(Y_1\wedge\cdots \wedge Y_k)\right)+\alpha([Y_1\wedge\cdots\wedge Y_k,Z]_{\overline{\mn}}).
\end{eqnarray*}
Application of this and Definition \ref{defder} yields after some straightforward calculations
\begin{eqnarray*}
&&\left(\partial\cdot Z\cdot \alpha\right)(Y_0\wedge\cdots\wedge Y_k)-\left(Z\cdot \partial\cdot \alpha\right)(Y_0\wedge\cdots\wedge Y_k)\\
&=&\sum_{i=0}^k(-1)^i(-1)^{|Y_i|(|Y_0|+\cdots+|Y_{i-1}|)}(-1)^{|Z|(|Y_0|+\cdots^{\hat{i}}\cdot+|Y_k|)}[Y_i,Z|_{\Mp}\cdot\left(\alpha(Y_0\wedge\cdots^{\hat{i}}\cdot\wedge Y_k)\right)\\
&&+\frac{1}{2}\sum_{i=1}^k(-1)^i(-1)^{|Y_i|(|Y_{i+1}+\cdots+|Y_k|)}\alpha([Y_0\wedge\cdots^{\hat{i}}\cdot Y_k,[Y_i,Z]_\Mp]_{\overline{\mn}})\\
&&+\frac{1}{2}\sum_{i=1}^k(-1)^i(-1)^{|Y_i|(|Y_{i+1}+\cdots+|Y_k|)}(-1)^{|Y_i||Z|}\alpha([Y_0\wedge\cdots^{\hat{i}}\cdot Y_k,Z]_\Mp,Y_i]_{\overline{\mn}}).
\end{eqnarray*}
Explicitly expanding the last term yields an expression equivalent to the second to last term, which shows the entire right-hand side above can be reduced to
\begin{eqnarray*}
\sum_{i=0}^k(-1)^i(-1)^{|Y_i|(|Y_{i+1}|+|Y_k|)}\left([Y_i,Z|_{\Mp}\cdot\alpha\right)(Y_0\wedge\cdots^{\hat{i}}\cdot\wedge Y_k)=\left(\sum_a \xi_a\wedge[\xi_a^\ddagger,Z]_{\Mp}\cdot \alpha\right)\left(Y_0\wedge\cdots\wedge Y_k\right),
\end{eqnarray*}
from which the desired relation follows.

Finally the property $\partial\circ\partial=0$ also follows from an immediate calculation. It could also be proved from the corresponding property of $\partial^\ast$ in Lemma \ref{coder1} and the fact $\partial$ and $\partial^\ast$ are adjoint operators with respect to a non-degenerate bilinear form on $C^k(\overline{\mn},\mV)$, such as the one on $C^k(\mn,\mW)$ in the subsequent Theorem \ref{formonchains}.
\end{proof}

\begin{definition}{\rm
The homology groups are the $\Mp$-modules defined as
{\rm$H_k(\mn,\mV)= \ker(\partial_k^\ast)/\mbox{im}(\partial_{k+1}^\ast).$} The cohomology groups are the $\ml$-modules defined as
{\rm $H^k(\overline{\mn},\mV)= \ker(\partial_k)/\mbox{im}(\partial_{k-1})$. }}
\end{definition}

BGG resolutions are always expressed in terms of such (co)homology groups, see the subsequent Theorem \ref{onlyHk}. Only if they are completely reducible as $\Mp$-modules this constitutes a resolution in terms of generalised Verma modules. The following Lemma follows immediately from equation \eqref{defcodiff}.

\begin{lemma}
\label{Hcompred}
The action of $\mn$ on the $\Mp$-module $H_k(\mn,\mV)$ is trivial, consequently $H_k(\mn,\mV)$ is completely reducible as a $\Mp$-module if it is completely reducible as an $\ml$-module.
\end{lemma}

In the following lemma, we derive the explicit form of the coboundary operator $\partial$ on the $C_\bullet({\mn},\mV)$-interpretation of $C^\bullet(\overline{\mn},\mV)$.
\begin{lemma}
\label{propstandder}
For any basis $\{\xi_a\}$ of $\mn$ with dual basis $\{\xi_a^\ddagger\}$ of $\overline{\mn}$ with respect to the Killing form, the coboundary operator satisfies
\begin{eqnarray*}
\partial(X\wedge f)&=&\frac{1}{2}\sum_a\xi_a\wedge [\xi_a^\ddagger,X]_{\mn}\wedge f-X\wedge\partial f
\end{eqnarray*}
for $X\in\mn$ and $f\in C_{k-1}({\mn},\mV)$ and $\partial v=\sum_a \xi_a\otimes \xi_a^\ddagger\cdot v$ for $v\in\mV=C_0({\mn},\mV)$.
\end{lemma}
\begin{proof}
It is readily checked that the expression
\begin{eqnarray*}
\partial\left(X_1\wedge\cdots\wedge X_k\otimes v\right)&=&\frac{1}{2}\sum_a\xi_a\wedge [\xi_a^\ddagger,X_1\wedge\cdots\wedge X_k]_{\mn}\otimes v+(-1)^k\sum_a X_1\wedge\cdots\wedge X_k \wedge \xi_a\otimes \xi_a^\ddagger\cdot v
\end{eqnarray*}
is equivalent with the one in Definition \ref{defder}. The result then follows easily.
\end{proof}

\subsection{Relation with $C^\bullet(\mn,\mW)$}

The super vector spaces of chains or cochains 
\[C_k(\overline{\mn},\mV^\ast)=\Lambda^k\overline{\mn}\otimes \mV^\ast=\mbox{Hom}(\Lambda^k\mn,\mV^\ast)=C^k(\mn,\mV^\ast)\] can be identified with the dual space of $C^k(\overline{\mn},\mV)=\Lambda^k\mn\otimes \mV$, where the pairing is induced from the Killing form and the pairing between $\mV$ and $\mV^\ast$. This is the subject of the following definition. 

\begin{definition}{\rm
\label{pairingchain}
The bilinear form $(\cdot,\cdot):\mV^\ast\times\mV\to\mC$ corresponds to the evaluation of $\mV^\ast$ on $\mV$: $(\alpha,v)=\alpha(v)$ for $\alpha\in\mV^\ast$ and $v\in\mV$. The bilinear form 
\begin{eqnarray*}
(\cdot,\cdot):\left(\otimes^k\overline{\mn}\otimes\mV^\ast\right)\times\left(\otimes^k\mn\otimes\mV\right)\to\mC
\end{eqnarray*}
is defined inductively by 
\begin{eqnarray*}
(Y\otimes q,X\otimes p)&=&(-1)^{|q||X|}(Y,X)(q,p)
\end{eqnarray*}
for $Y\in\overline{\mn}$, $X\in\mn$, $q\in\otimes^{k-1}\overline{\mn}\otimes\mV^\ast$, $p\in\otimes^{k-1}{\mn}\otimes\mV$ and $(Y,X)$ the evaluation of the Killing form in Proposition \ref{Killingexist}. The non-degenerate bilinear form $C^k(\mn,\mV^\ast)\times C^k(\overline{\mn},\mV)\to\mC$ is the restriction of the form above.}
\end{definition}
Just as the spaces $C^k(\overline{\mn},\mV)$ are naturally $\Mp$-modules, $C^k(\mn,\mW)$ are naturally $\Mp^\ast$-modules, so in particular $\ml$-modules. It is easily seen that the form in Definition \ref{pairingchain} is $\ml$-invariant, to formulate this invariance more broadly we need to introduce the following action.

For $A\in\mg$ and $f\in C_k({\mn},\mV)$, $A[f]\in C_k({\mn},\mV)$ is defined by induction as
\begin{eqnarray}
\label{gactionchain}
A[X\wedge g]=[A,X]_{\mn}\wedge g+(-1)^{|A||X|}X\wedge A[g]\quad\mbox{for }X\in\mn,\, g\in C_{k-1}({\mn},\mV)
\end{eqnarray}
and $A[v]=A\cdot v$ for $v\in\mV$. Note that for $A\in\Mp$, $A[f]=A\cdot f$ holds, but for general $A\in\mg$ this does not constitute a representation. If we restrict $A$ to be an element of $\Mp^\ast$, this introduces a $\Mp^\ast$-representation structure on $C_k({\mn},\mV)$, which corresponds to the identification of vector spaces given by
\[\left(\Lambda^k\mg\otimes\mV\right)/\left(\Lambda^k\Mp^\ast\otimes\mV\right)\cong\Lambda^k\mn\otimes\mV=C_k({\mn},\mV).\]
A similar definition of the action of $\mg$ on $C_k(\overline{\mn},\mV^\ast)$ follows immediately:
\begin{eqnarray}
\label{gactionchain2}
A[Y\wedge f]=[A,Y]_{\overline{\mn}}\wedge f+(-1)^{|A||Y|}Y\wedge A[f]\quad\mbox{for }Y\in\overline{\mn},\, f\in C^{k-1}(\mn,\mV^\ast)
\end{eqnarray}
and $A[w]=A\cdot w$ for $v\in\mV^\ast$.

The bilinear form in Definition \ref{pairingchain} satisfies 
\begin{equation}\label{contravarianceofform}(A[q],p)=-(-1)^{|A||q|}(q,A[p])\end{equation}
for $A\in\mg$, $q\in C^k(\mn,\mV^\ast)$ and $p\in C^k(\overline{\mn},\mV)$, which follows immediately form the invariance of the Killing form. In particular the bilinear pairing between the $\ml$-modules $C^\bullet(\mn,\mV^\ast)$ and $C^\bullet(\overline{\mn},\mV)$ is $\ml$-invariant.

Then we can define the operators $\delta^\ast: C_k(\overline{\mn},\mV^\ast)\to C_{k-1}(\overline{\mn},\mV^\ast)$ and $\delta:C^k(\mn,\mV^\ast)\to C^{k+1}(\mn,\mV^\ast)$ as the adjoint operators with respect to $(\cdot,\cdot)$ in Definition \ref{pairingchain} of respectively $\partial$ and $\partial^\ast$. By definition, $\delta$ and $\delta^\ast$ are $\ml$-module morphisms and satisfy $\delta\circ\delta=0=\delta^\ast\circ\delta^\ast$.

\begin{lemma}
\label{exprdelta}
The operator $\delta^\ast: C_k(\overline{\mn},\mV^\ast)\to C_{k-1}(\overline{\mn},\mV^\ast)$ satisfies
\begin{eqnarray*}
\delta^\ast(Y\wedge f)&=& -Y\cdot f-Y\wedge\delta^\ast(f)
\end{eqnarray*}
for $Y\in\overline{\mn}$ and $f\in C^{k-1}(\mn,\mV^\ast)$ and this is a $\Mp^\ast$-module morphism. The operator $\delta: C^k(\mn,\mV^\ast)\to C^{k+1}(\mn,\mV^\ast)$ satisfies
\begin{eqnarray*}
\delta\left(Y_1\wedge\cdots\wedge Y_k\otimes u\right)&=&\frac{1}{2}\sum_a(-1)^{|\xi_a|}\xi_a^\ddagger\wedge [\xi_a,Y_1\wedge\cdots\wedge Y_k]_{\overline{\mn}}\otimes u\\
&&+(-1)^k\sum_a (-1)^{|\xi_a|} Y_1\wedge\cdots\wedge Y_k \wedge \xi_a^\ddagger\otimes \xi_a\cdot u
\end{eqnarray*}
for $Y_j\in\overline{\mn}$ and $u\in\mV^\ast$.
\end{lemma}
\begin{proof}
The definition of the bilinear form in Definition \ref{pairingchain} implies that for $\beta\in C^k(\overline{\mn},\mV)=\mbox{ Hom}(\Lambda^k\overline{\mn},\mV)$ and for $Y_1\wedge\cdots\wedge Y_k\otimes u\in C^k(\mn,\mV^\ast)=\Lambda^k\overline{\mn}\otimes\mV^\ast$ it holds that
\begin{eqnarray*}
(Y_1\wedge\cdots\wedge Y_k\otimes u,\beta)&=&(-1)^{|u|(|Y_1|+\cdots+|Y_k|)}(u,\beta(Y_1\wedge\cdots\wedge Y_k)).
\end{eqnarray*}
From this and Definition \ref{defder} it follows that
\begin{eqnarray*}
\delta^\ast(Y_1\wedge\cdots\wedge Y_k\otimes u)&=&\sum_{i=1}^k(-1)^i(-1)^{|Y_i|(|Y_{i+1}+\cdots+ |Y_{k}|)}Y_1\wedge\cdots^{\hat{i}}\cdot Y_k\otimes Y_i\cdot u\\
&&-\frac{1}{2}\sum_{i=1}^k(-1)^{i}(-1)^{|Y_i|(|Y_{i+1}+\cdots+ |Y_{k}|)}[Y_1\wedge\cdots^{\hat{i}}\cdot Y_k,Y_i]\otimes u
\end{eqnarray*}
holds. Expanding the adjoint action in the second term yields the proposed formula.

The operator $\delta$ can be calculated similarly from the definition of $\partial^\ast$.
\end{proof}
The basis $\{(-1)^{\xi_a}\xi_a\}$ is a dual basis for $\mn$ of the basis $\{\xi^\ddagger_a\}$ for $\overline{\mn}$, with respect to the Killing form in Proposition \ref{Killingexist}, since this form is supersymmetric. Comparison with Definition \ref{defder} and Theorem \ref{propstandder} shows that $\delta$ and $\delta^\ast$ are naturally the coboundary operator and boundary operator for Lie superalgebra (co)homology of $\mn$ in irreducible $\mg$-modules. Next we prove that these operators are adjoint operators with respect to a hermitian form on $C^\bullet(\mn,\mV^\ast)$. For $\mg$ a Lie algebra, this form can always be chosen to be positive definite, see Section 3.5 in \cite{MR0142696}. For Lie superalgebras this is very exceptional, all such cases are classified in Section \ref{ExamGL}.
\begin{theorem}
\label{formonchains}
Consider a basic classical Lie superalgebra $\mg$, with parabolic subalgebra $\Mp=\ml+\mn$ and finite dimensional representation $\mW$. There is a non-degenerate hermitian form $\langle\cdot,\cdot\rangle$ on $C^\bullet(\mn,\mW)$ which is contravariant with respect to $\ml$, such that
\begin{eqnarray*}
\langle\delta f,g\rangle=-\langle f,\delta^\ast g\rangle
\end{eqnarray*}
holds, for $f\in C^{k-1}(\mn,\mW)$ and $g\in C^k(\mn,\mW)$.
\end{theorem}
\begin{proof}
First we consider a contravariant hermitian form $\langle\cdot,\cdot\rangle$ on $\mW$ as in Lemma \ref{contraform} for an adjoint operation $\dagger$. Next we define a nondegenerate hermitian form $\langle\cdot,\cdot\rangle$ on $\overline{\mn}$ by
\begin{eqnarray*}
\langle Y_1,Y_2\rangle=(Y_1^\dagger,Y_2)
\end{eqnarray*}
for $Y_1,Y_2\in\mn$ and $(\cdot,\cdot)$ the Killing form of Proposition \ref{Killingexist}. The fact that the form is hermitian follows from Lemma \ref{starwithKilling}. The invariance of the Killing form then implies that for $A\in\mg$,
\begin{eqnarray*}
\langle [A,Y_1]_{\overline{\mn}},Y_2\rangle=([Y_1^\dagger,A^\dagger],Y_2)=(Y_1^\dagger,[A^\dagger,Y_2])=\langle Y_1,[A^\dagger,Y_2]_{\overline{\mn}}\rangle.
\end{eqnarray*}
In particular this form is $\ml$-contravariant, $\langle [B,Y_1],Y_2\rangle=\langle Y_1,[B^\dagger,Y_2]\rangle$ for $B\in\ml$. Note that since $\mg$ is simple this hermitian form could also be constructed through using Lemma \ref{contraform}.

The combination of the two hermitian forms gives a hermitian form on each $C^k(\mn,\mW)$ which we denote by the same notation $\langle\cdot,\cdot\rangle$ and is defined inductively by
\begin{eqnarray*}
\langle Y_1\wedge f,Y_2\wedge g\rangle=\langle Y_1,Y_2\rangle \langle f,g\rangle.
\end{eqnarray*}
It then follows from a direct calculation, similar to the one in the proof of Lemma \ref{exprdelta} that $\delta$ and $-\delta^\ast$ are adjoint operators with respect to these hermitian forms. 
\end{proof}

The operators $\delta$ and $\delta^\ast$ also define (co)homology groups, $H^k(\mn,\mV^\ast)=\ker\delta_k/\mbox{im}\delta_{k-1}$ and $H_k(\overline{\mn},\mV^\ast)=\ker\delta^\ast_k/\mbox{im}\delta_{k+1}^\ast$. There exist clear relations between the four (co)homology groups, see also Theorem 17.6.1 in \cite{MR2906817}.
\begin{lemma}
\label{4groups}
As $\ml$-modules, the isomorphisms
\[ H^k({\mn},\mV^\ast)\cong \left(H_k(\overline{\mn},\mV^\ast)\right)^\vee \cong \left(\left(H^k(\overline{\mn},\mV)\right)^\ast \right)^\vee\cong \left(H_k(\mn,\mV)\right)^\ast  \]
hold with $M^\ast$ the dual representation and $M^\vee$ the twisted dual representation of the $\ml$-representation $M$ with respect to $\dagger$ as in Definition \ref{twisteddual}.
\end{lemma}
\begin{proof}
The first equality follows from the fact that $\delta^\ast$ and $-\delta$ are adjoint operators with respect to an $\ml$-contravariant hermitian form in Theorem \ref{formonchains}.

The property $H_k(\overline{\mn},\mV^\ast)\cong  \left(H^k(\overline{\mn},\mV)\right)^\ast$ follows from the fact that $\delta^\ast$ and $-\partial$ are adjoint operators with respect to a non-degenerate $\ml$-invariant bilinear form, in Definition \ref{pairingchain} and equation \eqref{contravarianceofform}. The equality of the first and fourth representation follows similarly.
\end{proof}
In particular this shows that in case $H_k(\mn,\mV)$ is completely reducible as an $\ml$-representation, it is isomorphic to the corresponding cohomology group $H^k(\overline{\mn},\mV)$.

\begin{remark}
If $\mV$ is not considered to be irreducible, the formulae above need to be adjusted to
\[ H_k(\mn,\mV^\vee)=H^k(\overline{\mn},\mV)^\vee. \]
\end{remark}

To end this section we briefly show that the homology groups are identical to certain Ext functors, which is a standard homological fact and explains the connection of Kostant cohomology with BGG resolutions.
\begin{lemma}
\label{cohomExt}
The cohomology group {\rm$H^k(\overline{\mn},\mV)=\ker\partial_k/\mbox{im}\partial_{k-1}$\rm} satisfies
{\rm\begin{eqnarray*}
H^k(\overline{\mn},\mV)&\cong& \mbox{Ext}^k_{\overline{\mn}}(\mC,\mV)
\end{eqnarray*}}as $\ml$-modules, with {\rm$\mbox{Ext}^k_{\overline{\mn}}(-,\mV)$} the $k$-th right derived functor of the left exact contravariant functor {\rm$\mbox{Hom}_{\overline{\mn}}(-,\mV)$}, see Section 2.5 and Chapter 3 in \cite{MR1269324}. 
\end{lemma}
\begin{proof}
The proof does not change substantially from the proof of the corresponding well-known fact for Lie algebra homology. It is based on the analog of the Koszul complex or Chevallay-Eilenberg complex, see Section 16.3 in \cite{MR2906817} or Section 7.7 in \cite{MR1269324}. 
\end{proof}

\subsection{The Kostant quabla operator for Lie superalgebras}

The Kostant quabla operator on $C^\bullet(\overline{\mn},\mV)$ is defined as
\begin{eqnarray}
\label{Kosquabla}
\Box=\partial\circ\partial^\ast+\partial^\ast\circ\partial&:& C^k(\overline{\mn},\mV)\to C^k(\overline{\mn},\mV).
\end{eqnarray}
The corresponding operator $\delta\delta^\ast+\delta^\ast\delta$ on $C^\bullet(\mn,\mW)$ is denoted by the same symbol $\Box$.

For Lie algebras, this operator is a central tool in the study of the (co)homology, see \cite{MR0142696}. For Lie superalgebras the role of this operator in the theory is not always as decisive, in this subsection we explore some useful properties of this operator in the setting of Lie superalgebras.

First we show that this operator can be expressed in terms of Casimir operators, as in Theorem 4.4 in \cite{MR0142696} for Lie algebras, or Lemma 4.4 in \cite{MR2036954} for $\mg=\mathfrak{gl}(m|n)$ and $\mathfrak{g}_0=\mg_{\overline{0}}=\mathfrak{gl}(m)\oplus\mathfrak{gl}(n)$.

\begin{theorem}
\label{Thmquabla}
Consider a basic classical Lie superalgebra $\mg$, with parabolic subalgebra $\Mp=\ml+\mn$ and representation $\mV$. The Kostant quabla operator on $C^k(\overline{\mn},\mV)$ takes the form
\begin{eqnarray*}
\iota(\Box f)=-\frac{1}{2}\left(\cC_2(\mV)+\sum_a \xi_a \xi_a^\ddagger-\sum_\kappa T_\kappa T_\kappa^\ddagger\right) \iota(f)=-\frac{1}{2}\left(\cC_2(\mV)-\cC_2+2\sum_a \xi_a \xi_a^\ddagger\right) \iota(f)
\end{eqnarray*}
where $\{\xi_a\}$ is a basis of $\mn$, $\{T_\kappa\}$ a basis of $\Mp^\ast$, $\iota$ is the natural embedding $\iota:\Lambda^k\mn\otimes\mV\,\hookrightarrow \,\Lambda^k\mg\otimes\mV$ and $\cC_2(\mV)$ the value of the quadratic Casimir operator $\cC_2=\sum_a\xi_a\xi^\ddagger_a +\sum_\kappa T_\kappa T^\ddagger_\kappa\in\cU(\mg)$ on $\mV$.
\end{theorem}
Before we prove this relation, we note that even though the right hand side is defined as an element of $\Lambda^k\mg\otimes\mV$, it follows from the theorem that it is contained in the subspace $C^k(\overline{\mn},\mV)$. It can also be directly checked that the expression of the element of $\cU(\mg)$ on the right hand side (which is equivalent to $\Box$) can be rewritten as an element of $\cU(\ml)$:
\begin{equation}\label{quablacenter}-\frac{1}{2}\left(\cC_2(\mV)+\sum_a[\xi_a,\xi_a^\ddagger]-\sum_bB_bB_b^{\ddagger}\right)\end{equation}
for $\{B_b\}$ a basis of $\ml$. This shows that in fact the quabla operator on the $\Mp$-module $C^\bullet(\overline{\mn},\mV)$ is equivalent to an element of $Z(\ml)$, the centre of $\cU(\ml)$.
\begin{proof}
We start by observing that the operator $\partial:C^k(\overline{\mn},\mV)\to C^{k+1}(\overline{\mn},\mV)$ can be rewritten as
\begin{eqnarray*}
\iota\left(\partial (X_1\wedge\cdots\wedge X_k\otimes v)\right)&=&\frac{1}{2}\sum_a\xi_a\wedge [\xi^\ddagger_a,X_1\wedge\cdots\wedge X_k]\otimes v-\frac{1}{2}\sum_\kappa T_\kappa\wedge [T^\ddagger_\kappa,X_1\wedge\cdots\wedge X_k]\otimes v\\
&&+\sum_a(-1)^{|\xi_a|(|X_1|+\cdots+|X_k|)} \xi_a\wedge X_1\wedge\cdots\wedge X_k\otimes \xi_a^\ddagger\cdot v,
\end{eqnarray*}
which follows from the proof of Lemma \ref{propstandder} and the relation
\begin{eqnarray*}
\sum_a \xi_a\wedge [\xi_a^\ddagger,X]_{\Mp^\ast}=\sum_\kappa T_\kappa\wedge [T_\kappa^{\ddagger},X]
\end{eqnarray*}
for $X\in\mn$ which follows from the completeness relations of dual bases and the invariance of the Killing form. 

Therefore we can rewrite the expression for $\partial$ (which is now considered to be acting on the $\mg$-module $\Lambda^\bullet \mg\otimes \mV$) further to
\begin{eqnarray*}
\partial&=&\frac{1}{2}\sum_a \xi_a\wedge\xi^\ddagger_a-\frac{1}{2}\sum_\kappa T_\kappa\wedge T_\kappa^\ddagger+\frac{1}{2}\sum_j A_j\wedge \left(A_j^\ddagger\right)^\mV
\end{eqnarray*}
with $\{A_j\}$ a basis of $\mg$. The notation $A^\mV$ is used for the action of $\mg$ on $\Lambda^\bullet \mg\otimes \mV$ which comes from the tensor product action on $\mV$ with trivial action on $\mg$, while the ordinary action of $\mg$ on $\Lambda\mg\otimes\mV$ is the tensor product action with the adjoint action on $\mg$. Thus we have obtained a new expression for an operator on $\Lambda^\bullet \mg\otimes \mV$ which corresponds to $\partial$ when restricted to $C^k(\overline{\mn},\mV)$.

The operator $\partial^\ast$ can be trivially extended to $\Lambda^\bullet \mg\otimes \mV$ by applying equation \eqref{defcodiff} for $A\in\mg$: $\partial^\ast_k(A\wedge f)=-A\cdot f- A\wedge \partial_{k-1}^\ast (f)$ for $f\in\Lambda^{k-1} \mg\otimes \mV$. Then it follows that $\partial^\ast$ is a $\mg$-module morphism. These observations lead to
\begin{eqnarray*}
\partial\partial^\ast+\partial^\ast\partial&=&-\frac{1}{2}\sum_a \xi_a\xi^\ddagger_a+\frac{1}{2}\sum_\kappa T_\kappa T_\kappa^\ddagger-\frac{1}{2}\sum_j A_j \left(A_j^\ddagger\right)^\mV\\
&&-\frac{1}{2}\sum_j A_j\wedge \partial^\ast\left(A_j^\ddagger\right)^\mV+\frac{1}{2}\sum_j A_j\wedge\left(A_j^\ddagger\right)^\mV \partial^\ast\\
&=&-\frac{1}{2}\left(\sum_a \xi_a\xi^\ddagger_a-\sum_\kappa T_\kappa T_\kappa^\ddagger+\cC_2(\mV)\right)+O
\end{eqnarray*}
with $O$ the operator on $\Lambda^\bullet \mg\otimes \mV$ which is defined by $O(B_1\wedge\cdots\wedge B_k\otimes v)=$
\begin{eqnarray*}
&&-\frac{1}{2}\sum_j(-1)^{|B||A_j|}[A_j,B_1\wedge\cdots\wedge B_k]\otimes A_j^\ddagger\cdot v\\
&&-\frac{1}{2}\sum_j(-1)^{|B||A_j|}\sum_i(-1)^{i+|B_i|(|B_{i+1}|+\cdots+|B_k|)}A_j\wedge B_1\wedge\cdots^{\hat{i}}\cdot\wedge B_k\otimes B_iA_j^\ddagger v\\
&&+\frac{1}{2}\sum_i(-1)^{i+|B_i|(|B_{i+1}|+\cdots+|B_k|)}\sum_j(-1)^{|B||A_j|+|B_i||A_j|}A_j\wedge B_1\wedge\cdots^{\hat{i}}\cdot\wedge B_k\otimes A_j^\ddagger B_iv
\end{eqnarray*} 
for $B_i\in \mg$, $v\in\mV$ and $|B|=|B_1|+\cdots+|B_k|$.

The property 
\begin{eqnarray*}
\sum_j\left( (-1)^{|A_j||X|}[A_j,X]\otimes A_j^\ddagger\cdot v+ A_j\otimes [A_j^\ddagger,X]\cdot v\right)=0
\end{eqnarray*}
for $X\in\mg$ follows from the invariance of the Killing form and shows that the operator $O$ is zero.
\end{proof}

We introduce the notation $\Box_k$ for the restriction of $\Box$ to $C^k(\overline{\mn},\mV)$. Denote by $C^k(\overline{\mn},\mV)_0$ the generalised eigenspace of $\Box_k$, with eigenvalue zero. For certain questions it will be important whether this $\ml$-module is completely reducible. Since $\Box\in Z(\ml)$, complete reducibility of $C^\bullet(\overline{\mn},\mV)_0$ implies that $C^\bullet(\overline{\mn},\mV)_0=\ker\Box$ holds, but we also prove that complete reducibility of $\ker\Box$ already implies $C^\bullet(\overline{\mn},\mV)_0=\ker\Box$ and therefore that $C^\bullet(\overline{\mn},\mV)_0$ is completely reducible.

\begin{lemma}
\label{compred}
If $\ker\Box$ on $C^k(\overline{\mn},\mV)$ is completely reducible as an $\ml$-module, then $C^k(\overline{\mn},\mV)_0=\ker\Box_k$.
\end{lemma}
\begin{proof}
If $C^k(\overline{\mn},\mV)_0\not=\ker\Box_k$, there is a non-zero highest weight vector in the $\ml$-module $C^k(\overline{\mn},\mV)_0/\ker\Box_k$, of the form $x+\ker\Box_k$ for an $x\in C^k(\overline{\mn},\mV)_0$. We denote the positive root vector of $\ml$ corresponding to root $\alpha$ by $X_\alpha$ and the negative root vector $(-1)^{|X_\alpha|}X_\alpha^\ddagger$ by $Y_\alpha$. It holds that $X_\alpha x\in \ker\Box_k$ for every positive root $\alpha$.

Equation \eqref{quablacenter} implies that the action of $\Box$ on a certain weight vector is given by $c+\sum_\alpha Y_\alpha X_\alpha$ for a constant $c$ depending on the weight of $x$. By definition of $x$, $\Box x= cx+y$ for $y=\sum_\alpha Y_\alpha X_\alpha x\in \ker\Box$. Since $x\in C^k(\overline{\mn},\mV)_0$ it follows quickly that $c=0$ and therefore $\Box^2x=\Box y=0$.

Now the vector $y=\sum_\alpha Y_\alpha X_\alpha x=\Box x$ is a highest weight vector in $\ker\Box$, by the relation $X_\beta y = \Box X_\beta x=0$ for every positive root $\beta$, but $y$ is also generated by negative root vectors. Therefore complete reducibility of $\ker\Box_k$ implies that $C^k(\overline{\mn},\mV)_0/\ker\Box_k=0$.
\end{proof}

As will be mentioned in the subsequent Lemma \ref{lemmaKostant}, in case $\partial$ and $\partial^\ast$ are disjoint operators there is an $\ml$-isomorphism between $H_k(\mn,\mV)$ and $\ker \Box_k$. In the following lemma we prove a weaker statement without making the disjointness assumption.

\begin{lemma}
\label{Boxcohom}
If $\ker\Box_k$ is a completely reducible $\ml$-module, there is an injective $\ml$-module morphism from $H_k(\mn,\mV)$ to $\ker \Box_k$.
\end{lemma}
\begin{proof}
Lemma \ref{compred} implies $C^k(\overline{\mn},\mV)_0=\ker\Box_k$. There is a unique module $A^k$ consisting of all generalised eigenspaces with eigenvalue different from zero which satisfies $C^k(\overline{\mn},\mV)=\ker\Box_k\oplus A^k$. This implies that 
\begin{eqnarray*}
\ker\partial_k^\ast&=&\left(\ker\partial_k^\ast\cap \ker\Box\right)\oplus \left(\ker\partial_k^\ast\cap A^k\right)\qquad\mbox{and}\\
\mbox{im}\partial_{k+1}^\ast&=&\left(\mbox{im}\partial_{k+1}^\ast\cap \ker\Box\right)\oplus \left(\mbox{im}\partial_{k+1}^\ast\cap A^k\right)
\end{eqnarray*}
as $\ml$-modules.

The operator $\Box$ is invertible on $A^k$. Since $\Box$ maps $\ker\partial^\ast_k$ to $\mbox{im}\partial^\ast_{k+1}\subset\ker\partial^\ast_k$ and is an element of $Z(\ml)$ it follows that $\left(\ker\partial_k^\ast\cap A^k\right)=\left(\mbox{im}\partial_{k+1}^\ast\cap A^k\right)$ and \[H_k(\mn,\mV)\cong \left(\ker\partial_k^\ast\cap \ker\Box\right)/\left(\mbox{im}\partial_{k+1}^\ast\cap \ker\Box\right)\].

The fact that $\ker\Box_k$ is completely reducible then implies there is an $\ml$-module $B^k\subset \ker\Box_k$ such that $\left(\ker\partial_k^\ast\cap \ker\Box\right)=B^k\oplus\left(\mbox{im}\partial_{k+1}^\ast\cap \ker\Box\right)$ and therefore $H_k(\mn,\mV)\cong B^k\subset \ker\Box_k$.
\end{proof}

The proof of this lemma can easily be adjusted to the case where $C^k(\overline{\mn},\mV)_0$ is not necessarily completely reducible. This yields the following result.

\begin{corollary}
\label{Boxcohom3}
Denote by $C^k(\overline{\mn},\mV)_0$ the generalised eigenspace of $\Box$ with eigenvalue zero. The homology group $H_k(\mn,\mV)$ satisfies {\rm
\begin{eqnarray*}
H_k(\mn,\mV)&=&\left(\mbox{ker}\partial^\ast_k\cap C^k(\overline{\mn},\mV)_0\right)/\left(\mbox{im}\partial_{k+1}^\ast\cap C^k(\overline{\mn},\mV)_0\right).
\end{eqnarray*}}
\end{corollary}

\subsection{Disjointness of the boundary and coboundary operator}

If $\partial$ and $\partial^\ast$ are disjoint operators, i.e. $\partial(\partial^\ast f)=0$ implies $\partial^\ast f=0$ and the same holds for the roles of $\partial$ and $\partial^\ast$ reversed, we obtain a harmonic theory. The disjointness property can also be expressed as
\begin{eqnarray*}
\ker\partial\cap\mbox{ im}\partial^\ast=0&\mbox{and}& \ker\partial^\ast\cap\mbox{ im}\partial=0.
\end{eqnarray*} 
This property will be essential in the construction of BGG resolutions in Section \ref{BGG}.

The following lemma follows immediately from the general theory in Proposition 2.1 in \cite{MR0142696}.
\begin{lemma}
\label{lemmaKostant}
If $\partial$ and $\partial^\ast$ given in Definition \ref{defder} are disjoint operators, the following decomposition of $\ml$-modules holds: {\rm
\begin{eqnarray}
C^k(\overline{\mn},\mV)&=&\mbox{im}\partial\,\oplus\, \ker \Box\,\oplus\,\mbox{im}\partial^\ast,
\end{eqnarray}}with $\Box$ as defined in equation \eqref{Kosquabla}. Moreover, it holds that {\rm$\ker \partial=\mbox{im}\partial\,\oplus\, \ker \Box$, $\,\ker\partial^\ast=\ker \Box\,\oplus\,\mbox{im}\partial^\ast$} and {\rm $\,\mbox{im}\Box=\mbox{im}\partial\oplus\mbox{im}\partial^\ast$} as $\ml$-submodules of $C^k(\overline{\mn},\mV)$. This implies that the following $\ml$-module isomorphisms exist: {\rm
\begin{eqnarray*}
H_k(\mn,\mV)\,\cong\, \ker\Box\,\cong\, H^k(\overline{\mn},\mV)
\end{eqnarray*}}with {\rm$H_k(\mn,\mV)=\ker\partial^\ast_k/\mbox{im}\partial^\ast_{k+1}$} and {\rm$H^k(\overline{\mn},\mV)=\ker \partial_k/\mbox{im}\partial_{k-1}$}. The same results hold for $\delta$, $\delta^\ast$ and $\Box$ on $C^\bullet(\mn,\mW)$.
\end{lemma}
Lemma \ref{4groups} already showed that in general $H_k(\mn,\mV)\cong \left(H^k(\overline{\mn},\mV)\right)^\vee$ holds, so in particular if $\partial^\ast$ and $\partial$ are disjoint the $\ml$-module $H_k(\mn,\mV)$ is its own twisted dual.

An important result is that the disjointness of $\partial$ and $\partial^\ast$ is equivalent with that of $\delta$ and $\delta^\ast$, which we prove next.
\begin{proposition}
\label{disjointiff}
The operators $\partial$ and $\partial^\ast$ on $C^\bullet(\overline{\mn},\mV)$ are disjoint if and only if the operators $\delta$ and $\delta^\ast$ on $C^\bullet({\mn},\mV^\ast)$ are disjoint.
\end{proposition}
\begin{proof}
First we extend the antilinear bijection $\psi_0:\mV\to\mV^\ast$ from Lemma \ref{bijection} to a bijection $\psi_k:C^k(\overline{\mn},\mV)\to C^k(\mn,\mV^\ast)$ iteratively by
\begin{eqnarray*}
\psi_k(X\wedge f)=(-1)^{|X||f|}X^\dagger\wedge \psi_{k-1}(f),
\end{eqnarray*}
for $X\in\mn$, $f\in C^{k-1}(\overline{\mn},\mV)$ and $\dagger$ the adjoint operation from Lemma \ref{bijection}. From this definition it follows by induction that $\psi_k(A[f])=-(-1)^{|A||f|}A^\dagger[\psi_k(f)]$ holds for $A\in\mg$ and $A[\cdot]$ as defined in equation \eqref{gactionchain} for $C^\bullet(\overline{\mn},\mV)$ and equation \eqref{gactionchain2} for $C^\bullet({\mn},\mV^\ast)$.

Then we calculate, using Lemma \ref{exprdelta} and the definition of $\psi_k$
\begin{eqnarray*}
\delta^\ast_k\circ\psi_k(X\wedge f)&=&-(-1)^{|X||f|}X^\dagger\cdot\psi_{k-1}(f)-(-1)^{|X||f|}X^\dagger\wedge \delta^\ast_{k-1}\circ \psi_{k-1}(f).
\end{eqnarray*}
On the other hand we calculate, using Definition \ref{defder} and the commutation relation between $\psi_k$ and $\mg$-action (which corresponds to the representations of $\Mp$ and $\Mp^\ast$ on $C^\bullet(\overline{\mn},\mV)$ and $C^\bullet({\mn},\mV^\ast)$ respectively when restricted) derived above that
\begin{eqnarray*}
\psi_{k-1}\circ\partial^\ast_k(X\wedge f)&=&(-1)^{|X||f|}X^\dagger\cdot\psi_{k-1}( f)-(-1)^{|X||f|}X^\dagger\wedge \psi_{k-2}\circ\partial_{k-1}^\ast(f).
\end{eqnarray*}
Together with $\delta^\ast_1\circ\psi_1(X\otimes v)=-\psi_0\circ\partial_0^\ast(X\otimes v)$ this yields the conclusion that $\delta^\ast_k\circ\psi_k=-\psi_{k-1}\circ\partial^\ast_k$ holds by induction.

Similarly one can calculate that $\psi_{k+1}\circ\partial_k=-\delta_k\circ\psi_k$ holds by observing that Lemma \ref{starwithKilling} implies that if the basis $\{\xi_a\}$ of $\mn$ has as dual basis $\{\xi^\ddagger_a\}$, then the basis $\{(\xi_a^\ddagger)^\dagger\}$ of $\mn$ has as dual basis $\{\xi_a^\dagger\}$ with respect to the Killing form in Proposition \ref{Killingexist}.

The statement follows immediately from these observations.
\end{proof}

In order to construct BGG resolutions in Section \ref{BGG} we need the condition $\mbox{im}\partial^\ast\cap\ker\Box=0$. We prove that this consequence of the disjointness of $\partial$ and $\partial^\ast$ (see Lemma \ref{lemmaKostant}) is in fact equivalent with this disjointness.

\begin{theorem}
\label{Boxcohom2}
Denote by $C^\bullet(\overline{\mn},\mV)_0$ the generalised eigenspace of $\Box$ with eigenvalue zero. The following statements are equivalent {\rm
\begin{eqnarray*}
(1)\,\quad\qquad \mbox{im}\partial^\ast\cap\ker\Box=0&& (5)\quad \mbox{im}\partial\cap\ker\Box=0\\
(2)\,\,\quad \mbox{im}\partial^\ast\cap C^\bullet(\overline{\mn},\mV)_0=0&& (6)\quad H^\bullet(\overline{\mn},\mV)\cong C^\bullet(\overline{\mn},\mV)_0\\
(3)\,\,\quad\quad C^\bullet(\overline{\mn},\mV)_0 \subset\ker\partial^\ast&& (7)\quad \partial\mbox{ and }\partial^\ast\mbox{ are disjoint}\\
(4)\quad H^\bullet(\overline{\mn},\mV)\cong C^\bullet(\overline{\mn},\mV)_0,
\end{eqnarray*}}where isomorphisms are isomorphisms of $\ml$-modules. The same statements hold for $\delta$, $\delta^\ast$ and $\Box$ on $C^\bullet(\mn,\mW)$.
\end{theorem}
\begin{proof}
If (2) holds, (1) follows trivially. The reverse implication follows from the fact that if $x\in\mbox{im}\partial^\ast$ satisfies $\Box^lx=0$ with $\Box^{l-1}x\not=0$, then $\Box^{l-1}x\in\mbox{im}\partial^\ast\cap\ker\Box$, therefore $\mbox{im}\partial^\ast\cap C^\bullet(\overline{\mn},\mV)_0\not=0$ implies $\mbox{im}\partial^\ast\cap\ker\Box\not=0$.

Since $\Box\in \cU(\ml)$ and $\partial^\ast$ is an $\ml$-module morphism any $x\in \mbox{im}\partial^\ast\cap C^\bullet(\overline{\mn},\mV)_0$ is of the form $\partial^\ast y$ with $y\in C^\bullet(\overline{\mn},\mV)_0$. This implies that (2) and (3) are equivalent.

Corollary \ref{Boxcohom3} shows that the combination of (2) and (3) is equivalent with (4).

Next we prove that (2) is equivalent with $\mbox{im}\partial\cap C^\bullet(\overline{\mn},\mV)_0=0$, from which the equivalences with $(5)$ and $(6)$ follow similarly to the first four. The operator $\Box$ is symmetric with respect to the non-degenerate contravariant hermitian form $\langle\cdot,\cdot\rangle$ on $C^\bullet(\mn,\mW)$ in Theorem \ref{formonchains}, therefore its different generalised eigenspaces are orthogonal with respect to each other. The form thus induces a non-degenerate form on $C^\bullet(\overline{\mn},\mV)_0$. If $f=\partial^\ast f_1\in\mbox{im}\partial^\ast\cap C^\bullet(\overline{\mn},\mV)_0$ different from zero, there is a $g\in C^\bullet(\overline{\mn},\mV)_0$ such that $\langle \partial^\ast f_1,g\rangle\not=0$ and therefore $\partial g\not=0$, while $\partial g\in\mbox{im}\partial\cap C^\bullet(\overline{\mn},\mV)_0$.

Finally Lemma \ref{lemmaKostant} shows that $(7)$ implies the other statements. The other direction can be proved as follows. Assume that $(1)$ (and therefore also $(5)$) holds. By definition of the quabla operator, $\mbox{im}\partial^\ast\cap \ker\partial\subset \ker\Box$ holds and therefore $\mbox{im}\partial^\ast\cap \ker\partial\subset \mbox{im}\partial^\ast\cap \ker\Box=0$. Since the same holds for the roles of $\partial$ and $\partial^\ast$ reversed, we obtain that $(1)$ and $(5)$ imply the disjointness of $\partial$ and $\partial^\ast$.
\end{proof}
Note that Theorem \ref{Boxcohom2} implies that the disjointness of $\partial$ and $\partial^\ast$ yields the property $\ker\Box=C^\bullet(\overline{\mn},\mV)_0$.

Since the subsequent Theorem \ref{onlyHk} states that proper BGG resolutions can only exist if $H_{\bullet}(\overline{\mn},\mW)$ is completely reducible, the following corollary will be useful.
\begin{corollary}
\label{Boxcohom4}
Assume the homology groups $H_{\bullet}(\overline{\mn},\mW)$ are completely reducible $\ml$-modules. The operators $\delta$ and $\delta^\ast$ are disjoint if and only if $H_{\bullet}(\overline{\mn},\mW)\cong \ker\Box$.
\end{corollary}
\begin{proof}
The disjointness of $\delta$ and $\delta^\ast$ always implies $H_{\bullet}(\overline{\mn},\mW)\cong \ker\Box$, see Lemma \ref{lemmaKostant}.

Now we assume that $H_{\bullet}(\overline{\mn},\mW)$ is completely reducible and $H_{\bullet}(\overline{\mn},\mW)\cong \ker\Box$. Lemma \ref{compred} then implies that $\ker\Box=C^\bullet(\mn,\mW)_0$ and therefore Theorem \ref{Boxcohom2} $(4)\leftrightarrow(7)$ can be applied.
\end{proof}

Finally we prove that, under the appropriate disjointness assumption and with the assumption that $H_{\bullet}(\overline{\mn},\mW)$ is completely reducible (as is required in order to have BGG resolutions, see the subsequent Theorem \ref{onlyHk}), the homology groups are easily characterised using the quadratic Casimir operator.
\begin{theorem}
\label{structure cohomgroups}
Assume that $\delta$ and $\delta^\ast$ are disjoint and $H_{\bullet}(\overline{\mn},\mW_\lambda)$ is completely reducible as an $\ml$-module. Then the homology group $H_k(\overline{\mn},\mW_\lambda)$ is isomorphic to the direct sum of the $\ml$-submodules of $C^k(\mn,\mW_\lambda)$ which have a highest weight $\mu$ such that $\cC_2(\mW_\lambda)=\cC_2(\mW_\mu)$ with $\cC_2$ the quadratic Casimir operator on $\mg$. Here $\mW_\mu$ is the (not necessarily finite dimensional) irreducible $\mg$-module with highest weight $\mu$.
\end{theorem}
\begin{proof}
Lemma \ref{lemmaKostant} and Theorem \ref{Boxcohom2} imply that $H_k(\overline{\mn},\mW_\lambda)\cong\ker\Box_k\cong C^k(\mn,\mW_\lambda)_0$. 

Equation \eqref{quablacenter} and the definition of $\delta$ and $\delta^\ast$ through the bilinear form in Definition \ref{pairingchain} imply that
\begin{eqnarray}
\label{Boxg0}
\Box&=&-\frac{1}{2}\left(\cC_2(\mW_\lambda)-\sum_pB_pB_p^\ddagger- \sum_a[\xi_a,\xi_a^\ddagger ]\right)
\end{eqnarray} 
holds for $\{B_p\}$ a basis of $\ml$, where the commutator $\sum_a[\xi_a,\xi_a^\ddagger ]$ is always an element of the Cartan subalgebra of $\mg$. The invariance of the Killing form implies that $\Box$ commutes with the action of $\ml$ and therefore acts as a scalar on an irreducible $\ml$-submodule of $\Lambda^k\overline{\mn}\otimes \mW_\lambda$.

Since $C^k(\mn,\mW_\lambda)_0$ is completely reducible we only need to calculate the eigenvalue of $\Box$ on a highest weight vector of $\ml$ in $C^k(\mn,\mW_\lambda)$ of weight $\mu$ and $H_k(\overline{\mn},\mW_\lambda)$ is isomorphic to the direct sum of the $\ml$-submodules of $C^k(\mn,\mW_\lambda)$ generated by highest weight vectors with eigenvalue zero. This can be calculated as in Proposition 5.6 and Theorem 5.7 in \cite{MR0142696} for Lie algebras, but we take a shortcut here.

The element of $\cU(\mg)$ given by $-\frac{1}{2}\left(\cC_2(\mW_\lambda)-\cC_2 \right)$ can be rewritten as
\begin{eqnarray*}
-\frac{1}{2}\left(\cC_2(\mW_\lambda)-\sum_pB_pB_p^\ddagger-2\sum_a (-1)^{|\xi_a|}\xi_a^\ddagger \xi_a- \sum_a[\xi_a,\xi_a^\ddagger ]\right).
\end{eqnarray*}
Comparing this expression with $\Box$ yields the fact that the eigenvalue of $\Box$ on an $\ml$-highest weight vector of weight $\mu$ is identical to the eigenvalue of $-\frac{1}{2}\left(\cC_2(\mW_\lambda)-\cC_2 \right)$ on a $\mg$-highest weight vector of weight $\mu$. This concludes the proof.
\end{proof}
In the specific case $\ml=\mathfrak{gl}(m)\oplus\mathfrak{gl}(n)$, where the extra condition on complete reducibility is not required (and the radical $\mn$ is abelian), this was proved in Lemma 4.5 in \cite{MR2036954}.

\section{Construction of BGG resolutions}
\label{BGG}
We continue to use the same notations and conventions as in the previous section and the introduction. The goal of this section is to construct resolutions of finite dimensional irreducible modules of Lie superalgebras in terms of generalised Verma modules, as in equation \eqref{resLep} for Lie algebras. 

\subsection{A necessary condition}

We prove in Theorem \ref{onlyHk} and Corollary \ref{onlyHk2} that if such a resolution exists, it is essentially given by a resolution induced by the homology groups. This immediately yields complete reducibility of $H_{\bullet}(\overline{\mn},\mW)$ as a necessary condition for BGG resolutions for $\mW$.
\begin{proposition}
\label{BottfromBGG}
Consider a basic classical Lie superalgebra $\mg$ with parabolic subalgebra $\Mp=\ml+\mn$. If there are $\ml$-modules $\{M_j,j\in\mN\}$ which are $\Mp$-modules with trivial $\mn$-action, for which there is a resolution of the irreducible finite dimensional $\mg$-module $\mW$ of the form
\begin{eqnarray*}
\cdots\rightarrow V^{M_k}\rightarrow V^{M_{k-1}}\rightarrow\cdots \rightarrow V^{M_1}\to V^{M_0}\to\mW\to0
\end{eqnarray*}
with $V^{M_j}=\cU(\mg)\otimes_{\cU(\Mp)}M_j$, then $H_k(\overline{\mn},\mW)\cong N_k/K_k$ with $N_k$ and $K_k$ submodules of $M_k$, where $K_k$ is also a quotient of $M_{k+1}$.
\end{proposition}
\begin{proof}
This follows from the definition of Ext functors and is exactly the property that was used in \cite{MR0578996} to derive Bott's result in \cite{Bott}, see also Theorem 6.6 in \cite{MR2428237}. It can also be proven directly as in Lemma 7.6 in \cite{MR0414645}.

Generalised Verma modules are projective modules in the category of finitely generated $\overline{\mn}$-modules, see Section 2.2 in \cite{MR1269324}. Therefore the proposed resolution
\begin{eqnarray*}
\cdots\rightarrow^{m_{k}} V^{M_k}\rightarrow^{m_{k-1}} V^{M_{k-1}}\rightarrow\cdots \rightarrow^{m_1} V^{M_1}\to^{m_0} V^{M_0}\to\mW\to0
\end{eqnarray*}
for certain $\mg$-module morphisms $\{m_j\}$, is a projective resolution and defines right derived functors of left exact contravariant functors, see Section 2.5 in \cite{MR1269324}. The functor $\mbox{Hom}_{\overline{\mn}}(-,\mC)$ is such a functor and maps the resolutions above to the complex
\begin{eqnarray*}
0\to\mbox{Hom}_{\overline{\mn}}(V^{M_0},\mC)\to^{\phi_0}\mbox{Hom}_{\overline{\mn}}(V^{M_1},\mC)\to^{\phi_1}\cdots\to^{\phi_{k-1}}\mbox{Hom}_{\overline{\mn}}(V^{M_k},\mC)\to^{\phi_k}\cdots
\end{eqnarray*}
after scratching $\mbox{Hom}_{\overline{\mn}}(\mW,\mC)$ and where the $\ml$-module morphisms $\{\phi_j\}$ are induced from the $\mg$-module morphisms $\{m_j\}$ as $\phi_j(x)=x\circ m_j$ for $x\in\mbox{Hom}_{\overline{\mn}}(V^{M_j},\mC)$.

Then by definition, the derived functors of $\mbox{Hom}_{\overline{\mn}}(-,\mC)$, which are denoted by $\mbox{Ext}^k_{\overline{\mn}}(-,\mC)$ for $k\ge 1$ evaluated in $\mW$ satisfy
\begin{eqnarray*}
\mbox{Ext}^k_{\overline{\mn}}(\mW,\mC)&=&\ker(\phi_k)/\mbox{im}(\phi_{k-1}),
\end{eqnarray*}
which is independent of the choice of projective resolution, see Lemma 2.4.1 in \cite{MR1269324}. Finally the proposition follows by 
\[\mbox{Ext}^k_{\overline{\mn}}(\mW,\mC)\cong\mbox{Ext}^k_{\overline{\mn}}(\mC,\mW^\ast)\cong H^k(\overline{\mn},\mW^\ast)\cong \left(H_k(\overline{\mn},\mW)\right)^\ast,\]
see Lemma \ref{cohomExt} and Lemma \ref{4groups}, and observing that $\mbox{Hom}_{\overline{\mn}}(V^{M},\mC)\cong {M^\ast}$ holds as $\ml$-modules for any $\ml$-module $M$.
\end{proof}

This leads to the conclusion that any BGG resolution (resolution in terms of direct sums of generalised Verma modules) must be in terms of the (co)homology groups.
\begin{theorem}
\label{onlyHk}
Consider a basic classical Lie superalgebra $\mg$ with parabolic subalgebra $\Mp=\ml+\mn$. If there are completely reducible $\ml$-modules $\{M_j,j\in\mN\}$ which are $\Mp$-modules with trivial $\mn$-action, for which there is a resolution of the irreducible finite dimensional $\mg$-module $\mW$ of the form
\begin{eqnarray*}
\cdots\rightarrow V^{M_k}\rightarrow V^{M_{k-1}}\rightarrow\cdots \rightarrow V^{M_1}\to V^{M_0}\to\mW\to0,
\end{eqnarray*}
then the homology groups $H_k(\overline{\mn},\mW)$ are completely reducible and subsequently $H_k(\overline{\mn},\mW)\cong H^k(\mn,\mW)$. Then $\mW$ has a resolution of the form
\begin{eqnarray*}
\cdots \to V^{H_k(\overline{\mn},\mW)}\to\cdots\to V^{H_1(\overline{\mn},\mW)}\to V^{H_0(\overline{\mn},\mW)}\to\mW\to0.
\end{eqnarray*}
\end{theorem}
\begin{proof}
The complete reducibility of $H_k(\overline{\mn},\mW)$ follows immediately from the complete reducibility of $M_k$ by Proposition \ref{BottfromBGG}. The equivalence of the homology group and homology group then follow immediately from Lemma \ref{4groups} since twisted dual representations of completely reducible representations are identical to the original representations.

Now assume that $H_\bullet(\overline{\mn},\mW)\not=M_\bullet$. With notations of the proof of Proposition \ref{BottfromBGG} this implies that at least one of the morphisms $\phi_j$ is not identically zero. This implies that im$(m_j)\not\subset$ $\overline{n}\cdot V^{M_{j-1}}$. Therefore there are isomorphic irreducible submodules of $M_j$ and $M_{j-1}$, denoted by $S_j$ and $S_{j-1}$ respectively, such that the highest weight vector of $V^{S_j}$ is mapped to the highest weight vector of $V^{S_{j-1}}$ by $m_j$. This implies that $V^{S_j}$ has trivial intersection with the kernel of $m_j$ and $V^{S_{j-1}}$ is in the image of $m_j$. Since we considered a resolution of $\mW$ this implies that $V^{S_j}$ has trivial intersection with the image of $m_{j+1}$ while $V^{S_{j-1}}$ is contained in the kernel of $m_{j-1}$. These two representations can therefore be removed from the resolution. By iterating this procedure we end up at a situation where $\phi_j\equiv0$ for all $j\in\mN$ and therefore $H_\bullet(\overline{\mn},\mW)\cong M_\bullet$.
\end{proof}

The resolutions in terms of direct sums of generalised Verma modules in equation \eqref{resLep} of \cite{MR0578996, lepowsky} possess the additional property that no highest weight vector is inside the kernel of the morphisms. This is equivalent to the fact that the complex does not decompose into a direct sum of complexes. We show that resolutions for basic classical Lie superalgebras in terms of modules induced by the homology groups automatically satisfy this property and moreover are singled out by this property.

\begin{corollary}
\label{onlyHk2}
Consider $\mg$, $\Mp=\ml+\mn$ as in Theorem \ref{onlyHk}. If the irreducible finite dimensional $\mg$-module $\mW$ has a resolution in terms of direct sums of generalised Verma modules, there is a resolution of the form
\begin{eqnarray*}
\cdots \to V^{H_k(\overline{\mn},\mW)}\to\cdots\to V^{H_1(\overline{\mn},\mW)}\to V^{H_0(\overline{\mn},\mW)}\to\mW\to0.
\end{eqnarray*}
This resolution does not decompose into a non-trivial direct sum of complexes and every resolution of $\mW$ with this property is of the form above.\end{corollary}
\begin{proof}
The procedure in the proof of Theorem \ref{onlyHk} immediately implies that every resolution can be made smaller until it is of the proposed form.

By Proposition \ref{BottfromBGG} a resolution by modules induced by the homology groups can not be made smaller.
\end{proof}

\subsection{A sufficient condition}

We have obtained that any proper BGG resolution is given in terms of modules induced by the (co)homology groups. The purpose of the remainder of this section is to construct resolutions in terms of these homology groups. If the homology group is completely reducible this yields the only BGG resolution. If the homology group is not completely reducible, the obtained resolution is the closest one can get to a proper BGG resolution.

\begin{definition}
\label{defjet}
\rm For a $\Mp$-module $\mF$, the coinduced $\mg$-module $\cJ(\mF)$ is defined as a vector space by 
\begin{eqnarray*}
\mbox{Hom}_{\cU(\Mp)}(\cU(\mg),\mF)&=&\{\alpha\in \mbox{Hom}(\cU(\mg),\mF)|\alpha(U Z)=-(-1)^{|U||Z|}Z\left(\alpha(U)\right)\quad\mbox{for }Z\in\Mp\mbox{ and }U\in\cU(\mg)\}.
\end{eqnarray*}
The action of $\mg$ on $\cJ(\mF)$ is defined as $(A\alpha)(U)=-(-1)^{|A||U|}\alpha(AU)$ for $\alpha\in \cJ(\mF)$, $A\in\mg$ and $U\in\cU(\mg)$, which makes $\cJ(\mF)$ a $\mg$-submodule of Hom$(\cU(\mg),\mF)$.
 \end{definition}

We will need these spaces for the case $\mF=C^k(\overline{\mn},\mV)$ and submodules and subquotients. Note that it is important to write the brackets, for instance, with $Z\in\Mp$, $U\in\cU(\mg)$ and $\alpha\in\cJ\left(\mF\right)$, 
\[(Z\alpha)(U)=-(-1)^{|Z||U|}\alpha(ZU)\,\mbox{ is not the same as }\, Z(\alpha(U))=-(-1)^{|Z||U|}\alpha(UZ).\]

Since the operator $\partial^\ast$ is $\Mp$-invariant, it immediately extends to a $\mg$-invariant operator $\cJ\left(C^k(\overline{\mn},\mV)\right)\to \cJ\left(C^{k-1}(\overline{\mn},\mV)\right)$ which we denote by the same symbol. The operator $\partial$ is not $\Mp$-invariant, see Lemma \ref{pactionder}, but it can be modified to a $\mg$-invariant operator $d:\cJ\left(C^k(\overline{\mn},\mV)\right)\to \cJ\left(C^{k+1}(\overline{\mn},\mV)\right)$. This is the subject of the following theorem. This corresponds to constructing a morphism between the first jet extension of $C^k(\overline{\mn},\mV)$ and $C^{k+1}(\overline{\mn},\mV)$, with scalar part equal to $\partial$. In the language of differential operators on parabolic geometries this first order operator is a twisted exterior derivative, as will be discussed in the Appendix.

\begin{theorem}
\label{thmdefop}
The operators
\begin{eqnarray*}
\partial^\ast:\cJ\left(C^k(\overline{\mn},\mV)\right)\to \cJ\left(C^{k-1}(\overline{\mn},\mV)\right)&& \left(\partial^\ast\alpha\right)(U)=\partial^\ast\left(\alpha(U)\right)\\
d:\cJ\left(C^k(\overline{\mn},\mV)\right)\to \cJ\left(C^{k+1}(\overline{\mn},\mV)\right)&&\left(d\alpha\right)(U)=\partial \left(\alpha(U)\right)+\sum_a(-1)^{|U||\xi_a|} \xi_a\wedge \alpha(U\xi_a^\ddagger)
\end{eqnarray*}
for a basis $\{\xi_a\}$ of $\mn$ with dual basis $\{\xi^\ddagger_a\}$ of $\overline{\mn}$ with respect to the Killing form in Proposition \ref{Killingexist}, for $\alpha\in \cJ\left(C^k(\overline{\mn},\mV)\right)$ and $U\in\cU(\mg)$, with action of $\partial$ and $\partial^\ast$ on $C^k(\overline{\mn},\mV)$ as defined in Definition \ref{defder} are $\mg$-module morphisms.
\end{theorem}
\begin{proof}
The operator $\partial^\ast:\cJ\left(C^k(\overline{\mn},\mV)\right)\to \mbox{Hom}(\cU(\mg),C^{k-1}(\overline{\mn},\mV))$ is clearly a $\mg$-module morphism. It remains to be checked that $\mbox{im}\partial^\ast\subset \cJ\left(C^{k-1}(\overline{\mn},\mV)\right)$. This follows from the fact that for $Z\in\Mp$ and $U\in\cU(\mg)$ it holds that
\begin{eqnarray*}&&(\partial^\ast\alpha)(UZ)=\partial^\ast\left(\alpha(UZ)\right)=-(-1)^{|U||Z|}\partial^\ast\left(Z\cdot\alpha(U)\right)\\
&=&-(-1)^{|U||Z|}Z\cdot\partial^\ast\left(\alpha(U)\right)=-(-1)^{|U||Z|}Z\cdot\left(\left(\partial^\ast\alpha\right)(U)\right).
\end{eqnarray*}

The fact that the operator $d:\cJ\left(C^k(\overline{\mn},\mV)\right)\to \mbox{Hom}\left(\cU(\mg),C^{k+1}(\overline{\mn},\mV)\right)$ is $\mg$-invariant is also trivial, it is the sum of two $\mg$-invariant operators acting between $\cJ\left(C^k(\overline{\mn},\mV)\right)$ and $\mbox{Hom}(\cU(\mg),C^{k+1}(\overline{\mn},\mV))$. We define the operator $T$ as the second term in the definition of $d$, so $d=\partial+T$.

We need to compute that the image of $d$ is contained in $ \cJ\left(C^{k+1}(\overline{\mn},\mV)\right)$. It follows from Definition \ref{defjet} and the invariance of the Killing form that
\begin{eqnarray*}
(T\alpha)(UZ)&=&\sum_a(-1)^{|Z||U|} \xi_a\wedge[\xi_a^\ddagger,Z]_\Mp\cdot \left(\alpha(U)\right)-\sum_a (-1)^{(|\xi_a|+|Z|)|U|}[Z,\xi_a]\wedge \alpha(U\xi_a^\ddagger)\\
&&-\sum_a(-1)^{|U|(|\xi_a|+|Z|)+|Z||\xi_a|} \xi_a\wedge Z\cdot(\alpha(U\xi_a^\ddagger))\\
&=&\sum_a(-1)^{|Z||U|} \xi_a\wedge[\xi_a^\ddagger,Z]_\Mp\cdot \left(\alpha(U)\right)- (-1)^{|U||Z|}Z\left((T\alpha)(U)\right)
\end{eqnarray*}
holds. Comparison with Lemma \ref{pactionder} then yields the proof.
\end{proof}
Note that $\cJ(\mbox{ker}\partial^\ast_{k})$ is identical to the kernel of $\partial^\ast_{k}$ on $\cJ(C^{k}(\overline{\mn},\mV))$ and likewise for the image.

The operator $d$ generates a coresolution of the module $\mV$ in terms of coinduced modules. This is the subject of the following theorem.
\begin{theorem}
\label{Rhamcomplex}
The sequence
\begin{eqnarray*}
0\to\mV\to^{\epsilon} \cJ(\mV)\to^{d_0}\cJ(C^1(\overline{\mn},\mV))\to^{d_1}\cdots\to^{d_k}\cJ(C^{k+1}(\overline{\mn},\mV))\to^{d_{k+1}}\cdots
\end{eqnarray*}
with $\epsilon$ defined by $\left(\epsilon(v)\right)(U)=S(U)v$ for $U\in\cU(\mg)$ with $S$ the principal anti-automorphism of $\cU(\mg)$ (the antipode as described in the Appendix) is a coresolution of $\mV$.
\end{theorem}
\begin{proof}
It needs to be proven that the sequence is a complex, $d_k\circ d_{k-1}=0$ and moreover, that this complex is exact, ker$d_k=$im$d_{k-1}$.

First we prove that im$\epsilon=\ker d_0$. From the definition of $\epsilon(v)$ and Lemma \ref{propstandder} it follows that $d_0(\epsilon(v))=0$. The relation $\mbox{im}\epsilon\cong \mV$ is immediate. Now if $\alpha\in \cJ(\mV)$ is in the kernel of $d_0$, then it follows quickly that $\alpha(UY)=-(-1)^{|U||Y|}Y( \alpha(U))$ for $Y\in \overline{\mn}$. Together with Definition \ref{defjet} this implies that $\alpha(UA)=-(-1)^{|U||A|}A( \alpha(U))$ for all $A\in\mg$ or $\alpha(U)=S(U)\left(\alpha(1)\right)$ for all $U\in\cU(\mg)$, so $ \ker d_0\cong\mV$.

The fact that $d_k\circ d_{k-1}=0$ follows from a direct calculation using Lemma \ref{propstandder}. Because of Theorem \ref{twisted} in the Appendix the exactness can be reduced to the case where $\mV=0$. Through the vector space isomorphism $\cJ(\Lambda^k\mn)\cong$ Hom$(\cU(\overline{\mn}),\Lambda^k\mn)$, obtained from the Poincar\'e-Birkhoff-Witt property, the exactness of $d$ corresponds to the exactness of its induced operator Hom$(\cU(\overline{\mn}),\Lambda^k\mn)\to$ Hom$(\cU(\overline{\mn}),\Lambda^{k+1}\mn)$. Theorem \ref{extder} shows that the homology of the operator $d$ then becomes equivalent with that of the formal exterior derivative on a flat supermanifold. This is known to have trivial homology, see \cite{MR0672426}.
\end{proof}

At this stage we need to point out that $C^k(\overline{\mn},\mV)$ has a gradation and corresponding finite filtration as a $\Mp$-module. The gradation follows from giving each simple root vector in $\mn$ degree one, the elements of $\ml$ degree zero and the corresponding finite filtration of $\mV$. This gradation is naturally inherited by Hom$(\cU(\mg),C^k(\overline{\mn},\mV))$ where the gradation of a homomorphism is given by the gradation of its values. The corresponding filtration on $\cJ(C^k(\overline{\mn},\mV))$ is defined by restriction. It will be very useful for the sequel that the operators $\partial^\ast$ and $\partial$ have degree zero, while the operator $T$ raises the degree by one. The $\mg$-invariant extension of the Kostant quabla operator $\Box$ is given by
\begin{eqnarray*}
\widetilde{\Box}&=&d\partial^\ast+\partial^\ast d=\Box+T\partial^\ast+\partial^\ast T \,:\,\, \cJ\left(C^k(\overline{\mn},\mV)\right)\to \cJ\left(C^k(\overline{\mn},\mV)\right)
\end{eqnarray*}
Since the operator $\Box_T=\Box-\widetilde{\Box}=-T\partial^\ast-\partial^\ast T$ raises degree on $\mbox{Hom}(\cU(\mg),C^k(\overline{\mn},\mV))$ it is nilpotent, this fact will be essential. Obviously also compositions of $\Box_T$ with operators of degree zero are nilpotent.


\begin{lemma}
\label{inversequabla}
If the condition {\rm $\mbox{im}\partial^\ast\cap\ker\Box=0$} holds on $C^\bullet(\overline{\mn},\mV)$, the operator $\widetilde{\Box}$ is invertible on {\rm $\cJ\left(\mbox{im}\partial^\ast\right)$}.
\end{lemma}
\begin{proof}
The operator $\Box$ on $C^\bullet(\overline{\mn},\mV)$ can be restricted to $\mbox{im}\partial^\ast$ since $\partial^\ast$ and $\Box$ commute. The condition {\rm $\mbox{im}\partial^\ast\cap\ker\Box=0$} then implies that $\Box$ is invertible on im$\partial^\ast$. Similarly, the operator $\widetilde{\Box}$ on $\cJ(C^\bullet(\overline{\mn},\mV))$ can be restricted to $\mbox{im}\partial^\ast$ since $\partial^\ast$ and $\widetilde{\Box}$ commute.The operator $\widetilde{\Box}$ is then invertible as the sum of an invertible operator with a nilpotent operator.
\end{proof}

\begin{definition}{\rm
\label{defPi}
If the condition {\rm $\mbox{im}\partial^\ast\cap\ker\Box=0$} holds, the $\mg$-invariant operators $\Pi_k:\cJ(C^k(\overline{\mn},\mV))\to \cJ(C^k(\overline{\mn},\mV))$ are defined as
\begin{eqnarray*}
\Pi_k=1\,-\,\,d_{k-1}\circ \widetilde{\Box}^{-1}\circ \partial^\ast_k\,\,-\,\,\widetilde{\Box}^{-1}\circ \partial^\ast_{k+1}\circ d_k.
\end{eqnarray*}}
\end{definition}
This is well-defined by Lemma \ref{inversequabla}. We introduce the notation $p_k$ for the $\mg$-invariant projection $\cJ(\mbox{ker}\partial^\ast_k)\to \cJ(H_k(\mn,\mV))$ corresponding to the $\Mp$-invariant projection $\mbox{ker}\partial^\ast_k\to H_k(\mn,\mV)$ which we denote by the same symbol $p_k$. 

The following properties of the operator $\Pi$ in Definition \ref{defPi} on $\cJ\left(C^\bullet(\overline{\mn},\mV)\right)$ are now straightforward to derive, see also Proposition 5.5 in \cite{MR1856258}.
\begin{lemma}
\label{propsplit}
The $\mg$-invariant operator $\Pi_k:\cJ(C^k(\overline{\mn},\mV))\to \cJ(C^k(\overline{\mn},\mV))$ satisfies
\begin{eqnarray*}
(1)\quad \Pi_k\circ\partial^\ast_{k+1}=0=\partial^\ast_k\circ\Pi_k&&(3)\quad \Pi_{k+1}\circ d_{k}=d_k\circ\Pi_k \\
(2)\quad p_k\circ \Pi_k=p_k \mbox{ on }\ker\partial^\ast_k&&(4)\quad \Pi_k^2=\Pi_k.
\end{eqnarray*}
\end{lemma}

Lemma \ref{propsplit} $(1)$ shows that $\Pi_k$ is in fact a $\mg$-module morphism from $\cJ(C^k(\overline{\mn},\mV))$ to  $\cJ(\ker\partial^\ast_k)$, Lemma \ref{propsplit} $(1)$ also implies that when $\Pi_k$ is restricted to a morphism from $\cJ(\ker\partial_k^\ast)$ to  $\cJ(\ker\partial^\ast_k)$, this naturally descends to a morphism from $\cJ(H_k(\mn,\mV))$ to  $\cJ(\ker\partial^\ast_k)$. This $\mg$-module morphism is denoted by
\[L_k\,:\,\, \cJ(H_k(\mn,\mV))\to\cJ(\ker\partial^\ast_k).\]
It follows from lemma \ref{propsplit} $(2)$ that $p_k\circ L_k=1$, so in particular $L_k$ is injective.

Now we come to the definition of the operator that will be responsible for the desired coresolution of $\mV$.
\begin{definition}{\rm
\label{defD}
The $\mg$-invariant operator $D_k: \cJ(H_k(\mn,\mV))\to  \cJ(H_{k+1}({\mn},\mV))$ is defined as
\begin{eqnarray*}
D_k&=&p_{k+1}\circ d_k\circ L_k.
\end{eqnarray*}}
\end{definition}
This is well defined since by Lemma \ref{propsplit} $(1)$ and $(3)$ $\mbox{im} (d_k\circ L_k)\subset\ker \partial^\ast_{k+1}$. From its definition and the fact that $\Pi_{k+1}=L_{k+1}\circ p_{k+1}$, Lemma \ref{propsplit} (3) and (4), it follows that $L_{k+1}\circ D_k=d_k\circ L_k$ holds.

Using the defined operators it is now possible to construct a smaller coresolution of $\mV$ out of the coresolution in Theorem \ref{Rhamcomplex}, that corresponds to the infinitesimal version of the classical dual BGG sequences, mentioned in Section \ref{classical}.
\begin{theorem}
\label{BGGseq}
If the condition {\rm $\mbox{im}\partial^\ast\cap\ker\Box=0$} holds on $C^\bullet(\overline{\mn},\mV)$, the sequence
\begin{eqnarray*}
0\to\mV\to^{\epsilon'} \cJ(H_0({\mn},\mV))\to^{D_0}\cJ(H_1({\mn},\mV))\to^{D_1}\cdots\to^{D_k}\cJ(H_{k+1}({\mn},\mV))\to^{D_{k+1}}\cdots,
\end{eqnarray*}
with $\epsilon'=p_0\circ\epsilon$ for {\rm
\[p_0:\cJ(\mV)=\ker\partial_0^\ast\to\ker\partial_0^\ast/\mbox{im}\partial_1^\ast=\cJ\left(H_0({\mn},\mV)\right)\]} and $\epsilon$ given in Theorem \ref{Rhamcomplex}, is exact.
\end{theorem}
\begin{proof}
First we prove that this forms a complex,
\begin{eqnarray*}
D_{k+1}\circ D_{k}&=&p_{k+2}\circ d_{k+1}\circ L_{k+1}\circ p_{k+1}\circ d_k\circ L_k=p_{k+2}\circ  d_{k+1}\circ \Pi_{k+1}\circ d_k\circ L_k\\
&=&p_{k+2}\circ  d_{k+1}\circ d_k\circ \Pi_{k}\circ L_k=p_{k+2}\circ d_{k+1}\circ d_k\circ L_k=0,
\end{eqnarray*}
where we have used Lemma \ref{propsplit} consecutively and Theorem \ref{Rhamcomplex}. The fact $D_0\circ\epsilon'=0$ follows similarly.

Now we prove the exactness of the complex. The fact that im$\epsilon'=\ker D_0$ follows from the property that $\epsilon'$ is injective (since $\mV$ is irreducible) and the fact that $L_0$ maps the kernel of $D_0$ injectively into the kernel of $d_0$ which is isomorphic to $\mV$, see Theorem \ref{Rhamcomplex}.

We consider an $f\in \cJ(H_{k+1}({\mn},\mV))$ that satisfies $D_{k+1}f=0$. Since $L_{k+2}\circ D_{k+1}=d_{k+1}\circ L_{k+1}$ it follows that $L_{k+1}(f)= d_{k}g$ for some $g\in \cJ(C^{k}(\overline{\mn},\mV))$, by Theorem \ref{Rhamcomplex}.

Define $g'=g-d_{k-1}\circ\widetilde{\Box}^{-1}\circ\partial^\ast_k(g)$, then $L_{k+1}(f)= d_{k}g=d_kg'$ and 
\[\partial^\ast_k (g')=\partial_k^\ast (g)-\partial^\ast_k\circ d_{k-1}\circ\widetilde{\Box}^{-1}\circ\partial^\ast_k(g)=0,\]
so $g'\in\cJ(\ker\partial^\ast)$. Since $d_{k}g'=L_{k+1}(f)$, $d_kg'$ is also inside $\cJ(\ker\partial^\ast)$ by Lemma \ref{propsplit} (1). We can then prove that the relation $g'=L_{k}\circ p_{k}(g')$ holds as follows. The element $g'-L_{k}\circ p_{k}(g')$ of $\cJ(\ker\partial^\ast)$ is inside $\cJ($im$\partial^\ast)$ since its projection onto $\cJ(H_{k}({\mn},\mV))$ is zero by the relation $p_{k}\circ L_k=1$. Because $\Box$ is invertible on $\mbox{im}\partial^\ast$ by Lemma \ref{lemmaKostant}, we can prove $g'-L_{k}\circ p_{k}(g')=0$ by calculating
\begin{eqnarray*}
\Box(g'-L_{k} p_{k}(g'))&=&(\partial^\ast d_k-\partial^\ast T_k)(g'-L_{k} p_{k}(g'))\\
&=&\partial^\ast L_{k+1}(f)-\partial^\ast L_{k+1} D_k p_{k}(g')-\partial^\ast T_k(g'-L_{k} p_{k}(g'))\\
&=&-\partial^\ast T_k(g'-L_{k}\circ p_{k}(g')),
\end{eqnarray*}
since $\partial^\ast L_{k+1}=0$ holds.

The operator $\Box$ is of degree zero while its action above strictly raises the degree, since this operator acts invertible this shows that $g'=L_{k}\circ p_{k}(g')$ holds.

Therefore we obtain that if $D_{k+1}f=0$ holds for $f\in \cJ(H_{k+1}({\mn},\mV))$, then $L_{k+1}(f)=L_{k+1}\circ D_k \circ p_k(g')$ holds. Since $L_{k+1}$ is injective, it follows that $f\in$ im$D_{k}$ and the theorem is proven.
\end{proof}

In the following central theorem we use the notation $V^\mF=\cU(\mg)\otimes_{\cU(\Mp)}\mF $ for any such $\mg$-module induced from a $\Mp$-module $\mF$. We only use the terminology (generalised) Verma module in case $\mF$ is irreducible. If $\mF$ is completely reducible $V^\mF$ corresponds to the direct sum of generalised Verma modules. If $\mF$ is an irreducible $\ml$-representation, in this context it is silently assumed to be a $\Mp$-module with trivial action of $\mn$.
\begin{theorem}
\label{finalthm}
Consider a basic classical Lie superalgebra $\mg$ with irreducible finite dimensional representation $\mW$ and parabolic subalgebra $\Mp=\ml+\mn$. If the operators $\delta$ and $\delta^\ast$ on $C^\bullet(\mn,\mW)$ of Lemma \ref{exprdelta} are disjoint, the module $\mW$ has a resolution in terms of induced modules
\begin{eqnarray*}
\cdots \to V^{H^k(\mn,\mW)}\to\cdots\to V^{H^1(\mn,\mW)}\to V^{H^0(\mn,\mW)}\to\mW\to0
\end{eqnarray*}
with homology group {\rm $H^k(\mn,\mW)=\ker\delta_{k}/\mbox{im}\delta_{k-1}$}. If the $\ml$-modules $H^k(\mn,\mW)$ are completely reducible, this is a resolutions in terms of generalised Verma modules.
\end{theorem}

\begin{proof}
First we show that there exists a non-degenerate bilinear $\mg$-invariant pairing between 
\[V^{\mF^\ast}=\cU(\mg)\otimes_{\cU(\Mp)}\mF^\ast \mbox{ and } \cJ(\mF)=\mbox{Hom}_{\cU(\Mp)}(\cU(\mg),\mF)\]
for any $\Mp$-module $\mF$. This pairing is induced by the $\mg$-invariant pairing between $\cU(\mg)\otimes\mF^\ast$ and $\mbox{Hom}(\cU(\mg),\mF)$ given by
\begin{eqnarray*}
(U\otimes\alpha,\phi)=(-1)^{|U||\alpha|}\alpha\left(\phi(U)\right)&\mbox{for}&U\in\cU(\mg), \,\alpha\in\mF^\ast\mbox{ and }\phi\in\mbox{Hom}(\cU(\mg),\mF).
\end{eqnarray*}
This pairing is clearly non-degenerate and in this context $\mg$-invariant means $(Au,\phi)=-(-1)^{|A||u|}(u,A\phi)$ for $A\in\mg$, $u\in\cU(\mg)\otimes\mF^\ast$ and $\phi\in\mbox{Hom}(\cU(\mg),\mF)$. Now we restrict the pairing to the submodule $\cJ(\mF)\subset\mbox{Hom}(\cU(\mg),\mF)$ in Definition \ref{defjet}. Since the invariance of the bilinear form implies that the elements in the submodule $K\subset \cU(\mg)\otimes\mF^\ast$ such that $\cU(\mg)\otimes_{\cU(\Mp)}\mF^\ast=\left(\cU(\mg)\otimes\mF^\ast\right)/K$ are orthogonal on $\cJ(\mF)$, the pairing then descends to $V^{\mF^\ast}\times \cJ(\mF)$. By construction, for each element of $\cJ(\mF)$ there is an element of $V^{\mF^\ast}$ which gives a non-zero pairing. The other way can be proved from the restriction of the original pairing to the subspaces $\cU(\overline{\mn})\otimes\mF^\ast$ and $\mbox{Hom}(\cU(\overline{\mn}),\mF)$.

 From this non-degenerate pairing it follows that if two subspaces of $V^{\mF^\ast}$ have the same space of orthogonal vectors inside $\cJ(\mF)$ they must coincide. 

Because of Proposition \ref{disjointiff} and Theorem \ref{Boxcohom2} we can apply Theorem \ref{BGGseq} for $\mV=\mW^\ast$. The $\mg$-invariant operators 
\[D^\ast_k:V^{H_k(\mn,\mV)^\ast}\to V^{H_{k-1}({\mn},\mV)^\ast}\]
defined by $(D_k^\ast u,\phi)=(u,D_k\phi)$ with $D_k$ the operator from Definition \ref{defD} and Theorem \ref{BGGseq}, for all $u\in\mV^{H_k(\mn,\mV)^\ast}$ and $\phi\in \cJ(H_k(\mn,\mV))$ then form an exact complex by the considerations above.

The proof then follows from the identification $H^j(\mn,\mW)\cong \left(H_j(\mn,\mV)\right)^\ast$ in Lemma \ref{4groups}.
\end{proof}

Now we restrict to the case where $H_{\bullet}(\overline{\mn},\mW)$ is completely reducible, which is a necessary condition in order to have BGG resolutions, according to Theorem \ref{onlyHk}.
\begin{corollary}
\label{finalthmCR} 
Consider a basic classical Lie superalgebra $\mg$ with irreducible finite dimensional representation $\mW$ and parabolic subalgebra $\Mp=\ml+\mn$. If the $\ml$-module $H_{\bullet}(\overline{\mn},\mW)$ is completely reducible and is isomorphic to $\ker\Box$, then $\mW$ can be resolved in terms of generalised Verma modules and the resolution is of the form
\begin{eqnarray*}
\cdots \to V^{H_k(\overline{\mn},\mW)}\to\cdots\to V^{H_1(\overline{\mn},\mW)}\to V^{H_0(\overline{\mn},\mW)}\to\mW\to0.
\end{eqnarray*}
\end{corollary}
\begin{proof}
Corollary \ref{Boxcohom4} implies that Theorem \ref{finalthm} can be applied. Lemma \ref{4groups} implies that $H^\bullet(\mn,\mW)\cong H_{\bullet}(\overline{\mn},\mW)$, therefore the exact complex of Theorem \ref{finalthm} is a resolution in terms of generalised Verma modules.
\end{proof}

Contrary to the classical case in \cite{MR0578996, lepowsky}, the resolution is not finite since the homolgy groups do not necessarily vanish (because there is no top exterior power for super vector spaces). For $\mathfrak{gl}(m|n)$, the method of super-duality $\mathfrak{gl}(m|n)\hookrightarrow\mathfrak{gl}(m|\infty)\leftrightarrow \mathfrak{gl}(m+\infty)$ gives a nice interpretation of this infinite behavior by identifying the weights of these highest weight vectors in $H_{j}(\overline{\mn},\mW)$ with orbits of the quotient of the Weyl group $W^1$ of the Lie algebra $\mathfrak{gl}(m+\infty)$, see e.g. \cite{Cheng}. If one ignores the superduality and only looks at the Lie superalgebra $\mathfrak{gl}(m|n)$, the highest weights come from both reflections of the Weyl group of $\mathfrak{g}$ or $\mathfrak{g}_{\overline{0}}$ as well as from so-called odd reflections, this also appears in the explicit example for $\osp$ in the subsequent Theorem \ref{BGGtaut}.

If a module has a BGG resolution, then the corresponding generalised Verma module has the property that its maximal submodule is generated by highest weight vectors, see also Corollary 5.2 in \cite{Cheng}. Since it only requires part of the BGG resolution we can formulate it more broadly.
\begin{corollary}
\label{submodVerma}
Consider a basic classical Lie superalgebra $\mg$, an irreducible finite dimensional representation $\mW$ and a parabolic subalgebra $\Mp=\ml+\mn$. If {\rm$\mbox{im}\delta^\ast_j\cap\ker\Box=0$} holds for $j\in\{1,2,\cdots,p\}$, then there is an exact complex of the form
\begin{eqnarray*}
 V^{H^p(\mn,\mW)}\to\cdots\to V^{H^1(\mn,\mW)}\to V^{H^0(\mn,\mW)}\to\mW\to0
\end{eqnarray*}
where $H^j(\mn,\mW)\cong \ker\Box_j$ for $j\in\{ 0,\cdots, p-1\}$.

Define $M(\lambda)$ as the irreducible $\ml$-module with highest weight $\lambda$ which is a $\Mp$-module with trivial $\mn$-action. If {\rm$\mbox{im}\delta^\ast_1\cap\ker\Box=0$} holds, which is identical to {\rm
\[\left(\overline{\mn}\cdot\mW_\lambda\right)\cap \ker\left(\sum_a(-1)^{|\xi_a|}\xi^\ddagger_a\xi_a\right)=0\]
}on $\mW_\lambda$ and if $H^1(\mn,\mW_\lambda)$ is completely reducible as an $\ml$-module, the unique maximal submodule of the generalised Verma module
\begin{eqnarray*}
V^{M(\lambda)}&=&\cU(\mg)\otimes_{\cU(\Mp)}M(\lambda)
\end{eqnarray*}
is generated by highest weight vectors in $V^{M(\lambda)}$.
\end{corollary}
\begin{proof}
We start from the exact complex, provided by Theorem \ref{Rhamcomplex}:
\begin{eqnarray*}
0\to\mW_\lambda^\ast\to^{\epsilon} \cJ(\mW_\lambda^\ast)\to^{d_0}\cJ(C^1(\overline{\mn},\mW_\lambda^\ast))\to^{d_1}\cdots\to^{d_{p-1}}\cJ(C^p(\overline{\mn},\mW_\lambda^\ast)).
\end{eqnarray*}
The proof of Proposition \ref{disjointiff} shows that the equivalence of the disjointness of the operators on $C^\bullet(\overline{\mn},\mV)$ and $C^\bullet(\mn,\mV^\ast)$ also follows from one another if they are restricted to act inside $\oplus_{k=0}^pC^k(\overline{\mn},\mV)$ and $\oplus_{k=0}^pC^k(\mn,\mV^\ast)$. These results can then be used to define $L_j:\cJ(H_j(\mn,\mW_\lambda^\ast))\to\cJ(\ker\partial^\ast_{j})$ for $j <p$ as underneath Lemma \ref{propsplit}, which leads to an exact complex as in Theorem \ref{BGGseq}
\begin{eqnarray*}
0\to\mW_\lambda^\ast\to^{\epsilon'} \cJ(H_0({\mn},\mW_\lambda^\ast))\to^{D_0}\cJ(H_1({\mn},\mW_\lambda^\ast)) \to^{D_1}\cdots\to^{D_{p-1}}\cJ(H_p({\mn},\mW_\lambda^\ast)).
\end{eqnarray*}
As in the proof of Theorem \ref{finalthm}, this leads through dualization to the exact complex
\begin{eqnarray*}
V^{H^p(\mn,\mW_\lambda)}\to \cdots\to V^{H^1(\mn,\mW_\lambda)}\to V^{H^0(\mn,\mW_\lambda)}\to\mW_\lambda\to0.
\end{eqnarray*}

The ideas in the proof of Theorem \ref{Boxcohom2}, Theorem \ref{finalthm} and the result in Corollary \ref{Boxcohom3} then show that for $j<p$, $H^j(\mn,\mW)\cong \ker\Box_j\cong H_{j}(\overline{\mn},\mW) $ holds.

The second part of the corollary follows from the first one with $p=1$ and the observation that $M(\lambda)\cong H^0(\mn,\mW_\lambda)$. This implies that $\mW_\lambda\cong  V^{M(\lambda)}/N$ with $N$ a subquotient of $V^{H^1(\mn,\mW_\lambda)}$, which is always a module generated by highest weight vectors if $H^1(\mn,\mW_\lambda)$ is completely reducible.
\end{proof}

In Example 5.1 in \cite{Cheng} it was proved that for $\mg=\mathfrak{gl}(1|2)$ and $\ml=\mh$ the Cartan subalgebra, the natural representation $\mC^{1|2}$ does not have a resolution in terms of Verma modules. This follows from the fact that the Verma module corresponding to $\mC^{1|2}$ does not satisfy the property that its maximal submodule is generated by highest weight vectors. The combination of Corollary \ref{submodVerma} with Theorem \ref{Boxcohom2} shows that for this case $H^0(\mn,\mC^{1|2})\not\cong\ker\Box_0$, since all modules are $\mh$-completely reducible. This is confirmed by the following example, which follows from a quick calculation.
\begin{example}
For $\mg=\mathfrak{gl}(m|n)$, $\ml=\mh$ and $\mW=\mC^{m|n}$ 
\[H^0(\mn,\mW)\cong \ker\Box_0\]
holds if and only if $m \ge n$.
\end{example} 

Finally we provide an example where the boundary and coboundary operator are not disjoint but where there exist BGG resolutions, showing that the sufficient condition obtained in Corollary \ref{finalthmCR} is not a necessary condition.
\begin{example}
\label{counter}
For $\mg=\mathfrak{osp}(1|2)$ and $\Mp=\mb$ the standard Borel subalgebra every finite dimensional module $\mW$ has a resolution in terms of Verma modules, while in general $\ker\Box\not\cong H_{\bullet}(\overline{\mn},\mW)$.
\end{example}
The resolution can be derived from the very simple structure of Verma modules of $\mathfrak{osp}(1|2)$ and is of the form
\begin{eqnarray*}
0\to V^{-\lambda-1}\to V^{\lambda}\to \mW_\lambda\to0
\end{eqnarray*}
for $\lambda\in\mN$. The fact $\ker\Box\not\cong H_{\bullet}(\overline{\mn},\mW)$ follows easily from considering $H_1(\overline{\mn},\mW)$, which is given by the product of the odd negative root vector of $\mathfrak{osp}(1|2)$ and the lowest weight vector of $\mW$. This shows that the product of the even negative root vector and the odd positive odd root vector acting on the lowest weight vector is also inside $\ker\Box_1$, by Theorem \ref{Thmquabla}.

Note that the BGG resolutions of $\mathfrak{osp}(1|2n)$ for arbitrary $n$ are derived in \cite{osp12n}.

\section{Type I Lie superalgebras and star representations}
\label{ExamGL}

In this section we prove that there is an extensive class of representations and parabolic subalgebras of $\mathfrak{gl}(m|n)$ and $\mathfrak{osp}(2|2n)$ that satisfy the sufficient conditions in Corollary \ref{finalthmCR} to construct BGG resolutions. This is done by studying their star representations (unitarisable representations), see Section \ref{KillingStar}. The only basic classical Lie superalgebras (excluding Lie algebras) which have finite dimensional star representations are $\mathfrak{sl}(m|n)$ ($\mathfrak{psl}(n|n)$ if $m=n$) and $\mathfrak{osp}(2|2n)$, see \cite{MR0424886, MR0787332}, these two are the ones of type $I$, the others are of type $II$, see \cite{MR0486011}. In \cite{MR1075732} and \cite{MR1067631} all unitarisable representations of these algebras were classified.

First we prove the following sufficient condition for disjointness of the boundary and coboundary operator.
\begin{theorem}
\label{thmstardisjoint}
Consider $\mg$ a basic classical Lie superlagebra of type $I$, with parabolic subalgebra $\Mp=\ml+\mn$ and finite dimensional irreducible representation $\mW$, which is a star representation with adjoint operation $\dagger$. If a basis $\{\xi_a\}$ of ${\mn}$ satisfies the property that $\xi_a^\dagger=(-1)^{|\xi_a|}\xi_a^\ddagger$ with $\{\xi_a^\ddagger\}$ the dual basis of $\mn$ with respect to a normalisation of the Killing form in Proposition \ref{Killingexist}, then the operators $\delta$ and $\delta^\ast$ are disjoint and the $\ml$-module $H_{\bullet}(\overline{\mn},\mW)$ is completely reducible.
\end{theorem}
\begin{proof}
Because of the assumed correspondence between the adjoint operation and the Killing form, the hermitian form on $\overline{\mn}$ in the proof of Theorem \ref{formonchains} is definite, since for a basis $\{\xi_a\}$ of $\mn$ it holds that
\begin{eqnarray*}
\langle \xi_a^\dagger,\xi_b^\dagger\rangle=(\xi_a,\xi_b^\dagger)=(-1)^{|\xi_b|}(\xi_a,\xi_b^\ddagger)=\delta_{ab}.
\end{eqnarray*}
Since $\mW$ is a star representation for $\dagger$, the contravariant hermitian form on $\mW$ from Theorem \ref{formonchains} is also positive definite. Therefore, the inner product $\langle\cdot,\cdot\rangle$ on $C^\bullet(\mn,\mW)$ from Theorem \ref{formonchains}, for which $\delta$ and $-\delta^\ast$ are adjoint operators, is positive definite. Because of this property it follows immediately that $\delta$ and $\delta^\ast$ are disjoint, since the relation $\langle\delta\delta^\ast f,f\rangle=-\langle \delta^\ast f,\delta^\ast f\rangle$ proves that $\delta\delta^\ast f=0$ implies $\delta^\ast f=0$ and similarly for the roles of $\delta^\ast$ and $\delta$ reversed.

Since the root systems of the basic classical Lie superalgebras of type I are multiplicity free, the adjoint operation can be restricted to an adjoint operation on $\ml$. This implies that $C^k(\mn,\mV)$ are star representations for $\ml$. Therefore $C^k(\mn,\mV)$ is a completely reducible $\ml$-module, so in particular $H_k(\overline{\mn},\mV)$ is.
\end{proof}

Now we will use this result to directly construct a class of BGG sequences for $\mathfrak{gl}(m|n)$, which extends the ones that were obtained in Remark 5.14 in \cite{MR2670923} (from tensor modules to general star representations) through the equivalence of categories related to super duality. Working with $\mathfrak{gl}(m|n)$ or its simple subalgebra $\mathfrak{sl}(m|n)$ makes no difference here. 

The natural module $\mC^{m|n}$ of $\mathfrak{gl}(m|n)$ or $\mathfrak{sl}(m|n)$ is a star representation, see \cite{MR1075732}. Therefore all its tensor powers are completely reducible, in fact the complete reducibility was already established in \cite{MR0884183} before the unitarisability was observed. All irreducible modules appearing as submodules of these tensor powers are called tensor modules. The highest weights of these modules are easily described in terms of Hook Young diagrams, see \cite{MR0884183}. Since the exterior powers do not have a top form for $\mC^{m|n}$ the module $\left(\mC^{m|n}\right)^\ast$ does not appear as a tensor module contrary to the classical case. This module also corresponds to a star representation but with different adjoint operation of $\mathfrak{gl}(m|n)$, see Proposition \ref{stardual}. All duals of tensor modules appear as submodules of the tensor powers of $\left(\mC^{m|n}\right)^\ast$.

These two types of star-representations are included in two strictly larger classes of star-representations of $\mathfrak{gl}(m|n)$ (corresponding to these two adjoint operations) which together constitute all star representations of $\mathfrak{gl}(m|n)$. In \cite{MR1075732} the explicit condition on the highest weight of an irreducible representation of $\mathfrak{gl}(m|n)$ to constitute a unitarisable representation was derived. We call these two classes the star representations of type $(1)$ (which include the tensor modules) and type $(2)$ respectively, corresponding to the terminology in \cite{MR1075732}.

\begin{theorem}
\label{mainthmgl}
For $\mg=\mathfrak{gl}(m|n)$ with a parabolic subalgebra $\Mp$, the $\mathfrak{g}$-module $\mV$ can be resolved in terms of direct sums of $\Mp$-Verma modules if
\begin{itemize}
\item $\mV$ is a star representation of type $(1)$ and $\Mp$ contains $\mathfrak{gl}(n)$
\item $\mV$ is a star representation of type $(2)$ and $\Mp$ contains $\mathfrak{gl}(m)$.
\end{itemize}
The explicit form of the resolution is given in Corollary \ref{onlyHk2}, with completely reducible homology groups $H_k(\overline{\mn},\mV)$ as described in Theorem \ref{structure cohomgroups}.
\end{theorem}
\begin{proof}
First we prove that the requirements in Theorem \ref{thmstardisjoint} are met. It suffices to prove this for one star representation of type $(1)$ and one of type $(2)$ (which we take as $\mC^{m|2n}$ and $\left(\mC^{m|2n}\right)^\ast$), since the conditions in Theorem \ref{thmstardisjoint} only depend on the adjoint operation.

The easiest realization of the $\mathfrak{gl}(m|n)$-representation on $\mC^{m|n}$ is as $m$ commuting variables and $n$ anti-commuting ones. These are denoted by $\{x_j|j=1,\cdots,m+n\}$ and satisfy commutation relations
\begin{eqnarray*}
x_ix_j=(-1)^{[i][j]}x_jx_i&\mbox{with}&[k]=0 \mbox{ (respectively $1$) if $k\le m$ (respectively $k>m$)}.
\end{eqnarray*}
The corresponding partial derivatives are then defined by the Leibniz rule $\partial_{x_i}x_j=\delta_{ij}+(-1)^{[i][j]}x_j\partial_{x_i}$. The differential operators $\{x_i\partial_{x_j}|i,j=1,\cdots,m+n\}$ generate a Lie superalgebra isomorphic to $\mathfrak{gl}(m|n)$ for the Lie bracket given by the super commutator. The inner product on $\mC^{m|n}$ is given by $\langle x_i,x_j\rangle=\delta_{ij}$, which implies $\mC^{m|n}$ is a star representation with adjoint operation
\begin{eqnarray*}
\left(x_i\partial_{x_j}\right)^\dagger=x_j\partial_{x_i}.
\end{eqnarray*}

The normalized Killing form on $\mathfrak{gl}(m|n)$ is given by 
\[(x_k\partial_{x_l},x_i\partial_{x_j})=C(-1)^{[k]}\delta_{kj}\delta_{li}\]
for arbitrary $x_k\partial_{x_l}$ and $x_i\partial_{x_j}$ and some constant $C$, which leads to $(x_b\partial_{x_a})^\ddagger=\frac{(-1)^{[b]}}{C}x_b\partial_{x_a}$.

The Borel subalgebra of $\mathfrak{gl}(m|n)$ is spanned by $\{x_i\partial_{x_j}|i\le j\}$. Saying that a parabolic subalgebra $\Mp$ contains $\mathfrak{gl}(n)$ is equivalent to saying that its radical $\mn$ is contained in the subalgebra $N=\{x_i\partial_{x_j}|i< j,i<m\}\subset\mathfrak{gl}(m|n)$. 

It then follows that for $C=1$, it holds that $(x_a\partial_{x_b})^\dagger=(-1)^{[a]+[b]}(x_a\partial_{x_b})^\ddagger$ for $x_a\partial_{x_b}\in N$. Theorem \ref{thmstardisjoint} therefore implies that $\delta$ and $\delta^\ast$ are disjoint operators for the considered cases, which implies that Theorem \ref{finalthm} can be applied. 

The reasoning for $V=\left(\mC^{m|n}\right)^\ast$ is similar. The adjoint operation on $\mathfrak{gl}(m|n)$ is given by $\left(x_i\partial_{x_j}\right)^\dagger=(-1)^{[i]+[j]}x_j\partial_{x_i}$, see Proposition \ref{stardual}, so normalization $C=-1$ has to be chosen for the Killing form.
\end{proof}
In particular if we take $\mV$ to be a tensor module and $\Mp$ the parabolic subalgebra with Levi algebra $\ml=\mg_{\overline{0}}=\mathfrak{gl}(m)\oplus\mathfrak{gl}(n)$, so $\mn\cong\mC^{0|mn}$, the necessary requirements are met. This corresponds to the main result in \cite{Cheng}. The case when $\mV$ is a tensor module and $\Mp$ contains $\mathfrak{gl}(n)$ can be derived from the equivalence of categories corresponding to the super duality
\[\mathfrak{gl}(m|n)\hookrightarrow \mathfrak{gl}(m|\infty)\leftrightarrow \mathfrak{gl}(m+\infty)\]
as is stated in Remark 5.14 in \cite{MR2670923}. 

We also note that in case $\mg=\mathfrak{sl}(4|n)$ and $\ml\cong \mathfrak{so}(4)\oplus\mathfrak{sl}(n)\oplus \mC\mathfrak{z}$ with $\mathfrak{z}=diag(N,N,N,N,4,\cdots,4)$, the structure corresponds to a complexification of the notion of super conformal geometry for which invariant differential operators and corresponding generalised Verma module morphisms have been thoroughly investigated in \cite{MR0905398}. In particular in \cite{MR0796630} the multiplets of generalised Verma modules were classified. The generalised Verma modules appearing in one BGG resolution must be contained inside one multiplet.

Also the Lie superalgebra $\mathfrak{osp}(2|2n)$ has two types of star representations, see \cite{MR0424886}, which have been classified in \cite{MR1067631}. These lead again to two classes for which we can construct BGG resolutions. Since the proof of this is so similar to the reasoning leading to Theorem \ref{mainthmgl}, we omit it.

\begin{theorem}
\label{mainthmC}
For $\mg=\mathfrak{osp}(2|2n)$ with a parabolic subalgebra $\Mp$, the $\mathfrak{g}$-module $\mV$ can be resolved in terms of direct sums of $\Mp$-Verma modules if
\begin{itemize}
\item $\mV$ is a star representation of type $(1)$ and $\Mp$ contains $\mathfrak{sp}(2n)$
\item $\mV$ is a star representation of type $(2)$.
\end{itemize}
The explicit form of the resolution is given in Corollary \ref{onlyHk2}, with completely reducible homology groups $H_k(\overline{\mn},\mV)$ as described in Theorem \ref{structure cohomgroups}.
\end{theorem}
Note that for $\mathfrak{osp}(2|2)\cong\mathfrak{sl}(1|2)$, this theorem agrees with Theorem \ref{mainthmgl}.

\begin{remark} The first statement in Theorem 6.3 is rather trivial since, as will be explained in Section
8, all finite dimensional representations for $\mathfrak{osp}(2|2n)$ possess such a resolution in terms of Kac modules.\end{remark}

\section{Orthosymplectic superalgebras and partial BGG resolutions}
\label{ExamOSp}

In this section we derive a simple criterion with weaker assumptions than Theorem \ref{thmstardisjoint}, which ensures the existence of BGG resolutions. In particular we derive some specific examples of BGG resolutions for $\mathfrak{osp}(m|2n)$. For some applications, see \cite{CSS, MR2180410}, only the last part of the BGG resolutions is required. We also focus on this feature for $\mathfrak{osp}(m|2n)$.

\begin{theorem}
\label{criterion}
\label{partialBGG}
Consider a basic classical Lie superalgebra $\mg$, with parabolic subalgebra $\Mp=\ml+\mn$ and a finite dimensional irreducible representation $\mW$. If for each $k\le p$ it holds that $\ker\Box_k$ is completely reducible and
\begin{eqnarray*}
[\ker\Box_{k-1}:M][\ker\Box_{k}:M]&\le &1
\end{eqnarray*}
for all irreducible $\ml$-modules $M$, then there are exact complexes of the form
\begin{eqnarray*}
&&0\to\mW^\ast\to^{\epsilon'} \cJ(H_0({\mn},\mW^\ast))\to^{D_0}\cJ(H_1({\mn},\mW^\ast))\to^{D_1}\cdots\to^{D_{p-1}}\cJ(H_{p}({\mn},\mW^\ast))\qquad\mbox{and}\\
&&\qquad\quad V^{H^p(\mn,\mW)}\to\cdots\to V^{H^1(\mn,\mW)}\to V^{H^0(\mn,\mW)}\to\mW\to0
\end{eqnarray*}
where the (co)homology groups with $H^k(\mn,\mW)\cong H_k(\overline{\mn},\mW)$ are completely reducible for $k\le p$ and isomorphic to $\ker\Box_k$ for $k<p$ and $H^p(\mn,\mW)$ is a quotient of $\ker\Box_p$. In particular if $p=\infty$ then $\mW$ has a resolution in terms of direct sums of generalised Verma modules, given in Theorem \ref{finalthm} with $H^\bullet(\mn,\mW)\cong H_{\bullet}(\overline{\mn},\mW)\cong \ker\Box$.
\end{theorem}
\begin{proof}
We focus on the case $p=\infty$, the case $p\in\mN$ can be proven similarly by using Corollary \ref{submodVerma} instead of Theorem \ref{finalthm}. The generalised zero eigenspace $C^\bullet(\mn,\mW)_0$ is equal to $\ker\Box$, see Lemma \ref{compred}.

Since the operators $\delta$ and $\delta^\ast$ commute with $\Box$ they act inside $\ker\Box=C^\bullet(\mn,\mW)_0\subset C^\bullet(\mn,\mW)$ and do not mix up different generalised eigenspaces. We will prove that $\delta$ and $\delta^\ast$ are disjoint operators when restricted to $\ker\Box$. Lemma \ref{lemmaKostant} applied to the space $\ker\Box\subset C^\bullet(\mn,\mW)$ then immediately implies that they are both identically zero on $\ker\Box$.

Since $\Box$ is symmetric with respect to the contravariant hermitian form in Theorem \ref{formonchains}, different generalised eigenspaces of $\Box$ are orthogonal with respect to each other.  the form is non-degenerate on $\ker\Box$. The operators $\delta$ and $\delta^\ast$ are both $\ml$-module morphism, so by the assumption on the multiplicities of $\ker\Box$, their image always consists of representations which appear with multiplicity one inside $\ker\Box$. Since non-isomorphic representations are orthogonal with respect to each other it then follows that the image of $\delta$ and $\delta^\ast$ in $\ker\Box$ are non-degenerate subspaces. 

Therefore, for $\delta^\ast f\in\ker\Box$ different from zero, there is an $\delta^\ast g\in\ker\Box$ such that $\langle\delta^\ast f,\delta^\ast g\rangle\not=0$. Since $-\delta$ and $\delta^\ast$ are adjoint operators, this implies $\delta\delta^\ast f\not=0$. The disjointness of $\delta$ and $\delta^\ast$ follows from this and therefore they are both identically zero on $\ker\Box$.

The existence of a resolution in terms of modules induced by the homology groups then follows from Theorem \ref{finalthm} and Theorem \ref{Boxcohom2} implies that the homology groups are completely reducible, so this is a resolution in terms of generalised Verma modules.
\end{proof}

With notations as in \cite{Tensor, MR2395482} we consider the positive root system of $\osp$ corresponding to the positive simple roots
\begin{eqnarray}
\label{choicepr}
\epsilon_1-\epsilon_2,\cdots, \epsilon_{d-1}-\epsilon_{d},\epsilon_d-\delta_1,\delta_1-\delta_2,\cdots,\delta_{n-1}-\delta_n,\begin{cases} \delta_n&m=2d+1\\ 2\delta_n& m=2d.\end{cases}
\end{eqnarray}
This positive root system has been used in e.g. \cite{Tensor, OSpHarm, CSS, MR2395482}. A finite dimensional irreducible representation, which has highest weight $\lambda$ with respect to this choice, will be denoted by $K_\lambda^{m|2n}$.

We choose the maximal parabolic subalgebra that corresponds to the first node $\epsilon_1-\epsilon_2$ of the Dynkin diagram. This implies that we must assume $m \ge 4$. For convenience we also assume $n>1$. This is the parabolic subalgebra considered in \cite{CSS, MR2395482}. In particular, it is clear from \cite{CSS} that a real form of this construction corresponds to a form of superconformal geometry, alternative to the one in \cite{MR0905398}.

The simple subalgebra corresponding to all simple roots in equation \eqref{choicepr} except the first one, is isomorphic to $\mathfrak{osp}(m-2|2n)$. The subalgebra $\ml$ which is the direct sum of $\mathfrak{osp}(m-2|2n)$ with the element of the Cartan subalgebra of $\osp$ corresponding to the first root is denoted by 
\[\mathfrak{cosp}(m-2|2n)\cong\mC\oplus\mathfrak{osp}(m-2|2n)\cong\{H_{\epsilon_1-\epsilon_2}\}\oplus\mathfrak{osp}(m-2|2n).\]

\begin{theorem}
\label{BGGtaut}
Consider $\mg=\osp$ and $\ml=\mathfrak{cosp}(m-2|2n)$, so $\mn\cong\mC^{m-2|2n}$. The natural representation $\mC^{m|2n}\cong K^{m|2n}_{\epsilon_1}$ has a resolution in terms of generalised Verma modules if $m-2n\le 1$. The explicit form of the resolution is given by
\begin{eqnarray*}
&&\cdots\to V^{M(-(d+j)\epsilon_1+2\epsilon_2+\epsilon_3+\cdots +\epsilon_d+(j+1)\delta_1)}\to\cdots\to V^{M(-(d+1)\epsilon_1+2\epsilon_2+\epsilon_3+\cdots +\epsilon_d+2\delta_1)}\to\\
&& V^{M(-d\epsilon_1+2\epsilon_2+\epsilon_3+\cdots +\epsilon_d+\delta_1)}\to V^{M(-(d-1)\epsilon_1+2\epsilon_2+\epsilon_3+\cdots +\epsilon_d)}\to\cdots\to V^{M(-(d-i)\epsilon_1+2\epsilon_2+\epsilon_3+\cdots +\epsilon_{d+1-i})}\\
&&\to\cdots\to V^{M(-2\epsilon_1+2\epsilon_2+\epsilon_3)}\to V^{M(-\epsilon_1+2\epsilon_2)}\to V^{M(\epsilon_1)}\to K^{m|2n}_{\epsilon_1}\to 0
\end{eqnarray*}
with $d=\lfloor m/2\rfloor$ and with $M(\lambda)$ the $\mathfrak{cosp}(m-2|2n)$-module with highest weight $\lambda$ which is a $\Mp$-module with trivial $\mn$-action.
\end{theorem}
\begin{proof}
The natural representation $\mC^{m|2n}\cong K^{m|2n}_{\epsilon_1}$ clearly decomposes a an $\mathfrak{osp}(m-2|2n)$-module as $K^{m-2|2n}_{0}\oplus K^{m-2|2n}_{\epsilon_2}\oplus K^{m-2|2n}_{0}$, see also Theorem 6.2 in \cite{OSpHarm}. As an $\ml=\mathfrak{cosp}(m-2|2n)$-module the decomposition is given by
\begin{eqnarray*}
K^{m|2n}_{\epsilon_1}&\cong & M(\epsilon_1)\,\oplus\,M(\epsilon_2)\,\oplus\, M(-\epsilon_1).
\end{eqnarray*}
As an $\ml$-module, $\overline{\mn}$ is isomorphic to $M(-\epsilon_1+\epsilon_2)$. From standard calculations of highest weight vectors, the fact that all of them have a different eigenvalue with respect to the quadratic Casimir operator, see \cite{MR0546778}, and Theorem 3 in \cite{Tensor} the decomposition of $C^k(\mn,K^{m|2n}_{\epsilon_1})$ can be obtained. With 
\begin{eqnarray*}
\kappa_k&=&\begin{cases}\epsilon_2+\cdots+\epsilon_{k+1} & \mbox{ if }k+1 \le d=\lfloor m/2\rfloor\\ \epsilon_2+\cdots+\epsilon_d+(k-d+1)\delta_1& \mbox{ if }k>d, \end{cases}\\
\mu_k&=&\begin{cases}2\epsilon_2+\epsilon_3+\cdots+\epsilon_{k+1} & \mbox{ if }k+1 \le d=\lfloor m/2\rfloor\\ 2\epsilon_2+\epsilon_3+\cdots+\epsilon_d+(k-d+1)\delta_1& \mbox{ if }k>d, \end{cases}
\end{eqnarray*}
it holds that
\begin{eqnarray*}
C^k(\mn,K^{m|2n}_{\epsilon_1})&\cong& M(-k\epsilon_1+\kappa_k)\otimes \left(M(\epsilon_1)\oplus M(\epsilon_2)\oplus M(-\epsilon_1)\right)\\
&=&M((1-k)\epsilon_1+\kappa_k)\oplus M(-(1+k)\epsilon_1+\kappa_k)\oplus M(-k\epsilon_1+\kappa_{k+1})\\
&&\oplus M(-k\epsilon_1+\mu_k)\oplus M(-k\epsilon_1+\kappa_{k-1})
\end{eqnarray*}
as $\ml$-representations. In particular $C^k(\mn,K^{m|2n}_{\epsilon_1})$ is completely reducible and multiplicity free for every $k\in\mN$. Therefore Theorem \ref{criterion} states that the BGG resolution exists and Corollary \ref{onlyHk2} determines the form of the resolution.

In order to calculate the homology groups $H_k(\overline{\mn},K^{m|2n}_{\epsilon_1})$ we apply Theorem \ref{structure cohomgroups}. Using once again the expression of the eigenvalues of the quadratic Casimir operator in \cite{MR0546778}, it follows that $H_k(\overline{\mn},K^{m|2n}_{\epsilon_1})\cong M(-k\epsilon_1+\mu_k)$.
\end{proof}

The following corollary is a generalisation of Example 3.5 in \cite{MR1856258}.
\begin{corollary}
With notations as in Theorem \ref{BGGtaut} and for $m-2n\le 1$ it holds that $H_0(\mn,K^{m|2n}_{\epsilon_1})={M(\epsilon_1)}$,
\begin{eqnarray*}
H^k(\mn,K^{m|2n}_{\epsilon_1})&=&{M(-k\epsilon_1+2\epsilon_2+\epsilon_3+\cdots+\epsilon_{k+1})} \quad\mbox{ for }k=1,\cdots,d-1\quad\mbox{ and}\\
H^{d+i}(\mn,K^{m|2n}_{\epsilon_1})&=&{M(-d\epsilon_1+2\epsilon_2+\epsilon_3+\cdots+\epsilon_{d}+(i+1)\delta_1)} \quad\mbox{ for }i\ge 0
\end{eqnarray*}
and $H_k(\overline{\mn},K^{m|2n}_{\epsilon_1})\cong H^k(\mn,K^{m|2n}_{\epsilon_1})$.
\end{corollary}

Another interesting case for $\osp$ is the adjoint representation, which is self-dual,
\[\mV=\osp=L^{m|2n}_{2\delta_1}=K^{m|2n}_{\epsilon_1+\epsilon_2}=\mV^\ast,\]
and its Cartan powers, see \cite{CSS}. In order to study the algebra of symmetries of the superconformal Laplace operator (see \cite{OSpHarm, MR2852289, CSS, MR2395482}) as an intertwining operator in Section 7 in \cite{CSS}, the first part of the BGG sequence in Theorem \ref{BGGseq} is needed for $\mV=L^{m|2n}_{2r\delta_1}=K^{m|2n}_{r\epsilon_1+r\epsilon_2}=\mV^\ast$.

\begin{theorem}
\label{forLapl}
Consider $\mg=\osp$ and $\ml=\mathfrak{cosp}(m-2|2n)$, with choice of positive simple roots as in equation \eqref{choicepr} for $m-2n\not=2$. The maximal submodule of the generalised Verma module with highest weight $\epsilon_1+\epsilon_2$ $V^{M(\epsilon_1+\epsilon_2)}$ is generated by a highest weight vector of weight and there is an exact complex
\begin{eqnarray*}
0\to K^{m|2n}_{\epsilon_1+\epsilon_2} \to \cJ(M(\epsilon_1+\epsilon_2))\to^{D_0} \cJ(M(2\epsilon_2))
\end{eqnarray*}
for some operator $\mg$-invariant operator $D_0$, which implies $K^{m|2n}_{\epsilon_1+\epsilon_2}\cong \ker(D_0)$.
\end{theorem}
\begin{proof}
In this proof we assume $m\ge 8$, only some unimportant notations of highest weights change when we would consider $4\le m\le 7$. The decomposition of $\osp=\overline{\mn}\oplus\ml\oplus \mn$ as an $\ml=\mathfrak{cosp}(m-2|2n)$-module is given by
\begin{eqnarray*}
K^{m|2n}_{\epsilon_1+\epsilon_2}&\cong&M(-\epsilon_1+\epsilon_2)\oplus \left(M(0)\oplus M(\epsilon_2+\epsilon_3) \right)\oplus M(\epsilon_1+\epsilon_2).
\end{eqnarray*}
Then it follows from Lemma 3.2 in \cite{CSS} that if $m\not=2+2n$
\begin{eqnarray*}
C^1(\mn,K^{m|2n}_{\epsilon_1+\epsilon_2})&=&M(-2\epsilon_1+2\epsilon_2)\oplus M(-2\epsilon_1+\epsilon_2+\epsilon_3)\oplus M(-2\epsilon_1)\oplus M(-\epsilon_1+\epsilon_2)\oplus M(2\epsilon_2)\\
&&\oplus M(\epsilon_2+\epsilon_3)\oplus M(0)\oplus\, \left(M(-\epsilon_1+\epsilon_2)\otimes M(\epsilon_2+\epsilon_3)\right).
\end{eqnarray*}
By calculating the highest weight vectors and the corresponding eigenvalues of the quadratic Casimir operator it follows from Corollary 1 in \cite{Tensor} that 
\begin{eqnarray*}
M(-\epsilon_1+\epsilon_2)\otimes M(\epsilon_2+\epsilon_3)&\cong&M(-\epsilon_1+2\epsilon_2+\epsilon_3)\oplus M(-\epsilon_1+\epsilon_2+\epsilon_3+\epsilon_4)\oplus M(-\epsilon_1+\epsilon_2).
\end{eqnarray*}
if $M\not=3$. Thus $C^1(\mn,K^{m|2n}_{\epsilon_1+\epsilon_2})$ is completely reducible and multiplicity free as a $\mathfrak{cosp}(m-2n|2n)$-module if $m-2n\not\in \{2,3\}$ and Theorem \ref{partialBGG} can be applied. This states that the homology group $H^1(\mn,K^{m|2n}_{\epsilon_1+\epsilon_2})$ is a quotient of $\ker\Box_1= M(2\epsilon_2)$, so either zero or isomorphic to $M(2\epsilon_2)$. Theorem \ref{partialBGG} also sates that the proposed part of the BGG resolution exists, which immediately implies that $H^1(\mn,K^{m|2n}_{\epsilon_1+\epsilon_2})$ can not be zero since otherwise $K^{m|2n}_{\epsilon_1+\epsilon_2} \cong \cJ(M(\epsilon_1+\epsilon_2))$ would hold.

Finally we focus on the case $m-2n=3$. The difference is that $C^1(\mn,K^{m|2n}_{\epsilon_1+\epsilon_2})$ might be not completely reducible. If $m-2n=3$,
\[M(-\epsilon_1+\epsilon_2)\otimes M(\epsilon_2+\epsilon_3)=M(-\epsilon_1+\epsilon_2+\epsilon_3+\epsilon_4)\oplus R\]
with $R$ a module with two highest weight vectors. The highest weight vectors are eigenvectors of $\Box$ by equation \eqref{Boxg0}. If there was a vector in $R$ belonging to generalised eigenspace with eigenvalue zero, this would generate a submodule of $R$ which must contain a highest weight vector, which is impossible this vector would have eigenvalue zero. Therefore $R$ has trivial intersection with the generalised eigenspace of $\Box$ with eigenvalue zero and Theorem \ref{partialBGG} can be applied.
\end{proof}

\begin{corollary}
\label{forLapl2}
Consider $\mg=\mathfrak{osp}(2n+2|2n)$ and $\ml=\mathfrak{cosp}(2n|2n)$, with choice of positive simple roots as in equation \eqref{choicepr}. There is an exact complex
\begin{eqnarray*}
0\to K^{2n+2|2n}_{\epsilon_1+\epsilon_2} \to \cJ(M(\epsilon_1+\epsilon_2))\to^{D_0} \cJ(Q)
\end{eqnarray*}
with $Q$ a quotient of $M(\epsilon_2)\odot M(\epsilon_2)$, which implies $K^{2n+2|2n}_{\epsilon_1+\epsilon_2}\cong \ker(D_0)$.
\end{corollary}
\begin{proof}
The non-complete reducibility of $C^1(\mn,K^{2n+2|2n}_{\epsilon_1+\epsilon_1})$ is located in the generalised eigenspace with eigenvalue zero, therefore Theorem \ref{partialBGG} is not applicable. Since we are mainly interested in the sequence on the co-induced modules we immediately work with $C^1(\overline{\mn}, K^{2n+2|2n}_{\epsilon_1+\epsilon_1})$.

The result in the proof of Theorem \ref{forLapl} show that
\begin{eqnarray*}
C^1(\overline{\mn},K^{2n+2|2n}_{\epsilon_1+\epsilon_2})&=&M(2\epsilon_1)\otimes\left(M(\epsilon_2)\odot M(\epsilon_2)\right)\oplus M(2\epsilon_1+\epsilon_2+\epsilon_3)\oplus \left(M(\epsilon_2)\odot M(\epsilon_2)\right)\oplus M(\epsilon_2+\epsilon_3)\\
&&\oplus M(\epsilon_1+2\epsilon_2+\epsilon_3)\oplus M(\epsilon_1+\epsilon_2+\epsilon_3+\epsilon_4)\oplus M(\epsilon_1+\epsilon_2),
\end{eqnarray*}
where $M(\epsilon_2)\odot M(\epsilon_2)$ is an indecomposable module, see \cite{OSpHarm} having two submodules: 
\[M(\epsilon_2)\odot M(\epsilon_2)\supset L\supset M(0).\] 
Here, $L$ is an indecomposable highest weight module, such that $M(\epsilon_2)\odot M(\epsilon_2)/L\cong M(0)$ and $L/M(0)\cong M(2\epsilon_2)$. The fact that $\mbox{im}\partial^\ast\cap\ker\Box_0=0$ can still be proven as follows. If $\mbox{im}\partial^\ast\cap\ker\Box_0\not=0$ there must be a vector in $M(-\epsilon_1+\epsilon_2)$ which is the image of a vector in $C^1(\overline{\mn},K^{2n+2|2n}_{\epsilon_1+\epsilon_2})_0=M(\epsilon_2)\odot M(\epsilon_2)$ under an $\ml$-morphism, which is impossible

The result then follows from the reasoning leading to Theorem \ref{BGGseq} and the combination of Corollary \ref{Boxcohom3} with $\mbox{ker}\partial^\ast\cap C^1(\overline{\mn},\mV)_0=M(\epsilon_2)\odot M(\epsilon_2)$.
\end{proof}

The results in Theorem \ref{forLapl} and Corollary \ref{forLapl2} perfectly agree with the conjectured $\osp$-representation structure of the space of conformal Killing vector fields in Section 7 of \cite{CSS}.

\section{Type I Lie superalgebras: typical and singly atypical modules}
\label{sectyp}
The basic classical Lie superalgebras of type $I$ have a $\mZ$-gradation of the form $\mg=\mg_{-1}+\mg_0+\mg_1$, with $\mg_{\overline{0}}=\mg_0$ and $\mg_{\overline{1}}=\mg_{-1}+\mg_1$. The set of highest weights of irreducible finite dimensional $\mg$-representations is the same as the corresponding set for $\mg_0$. Denote by $\mV^0_\lambda$ the irreducible $\mg_0$-representation with highest weight $\lambda$, the corresponding Kac module is defined by
\[K_\lambda=\cU(\mg)\otimes_{\cU(\mg_0+\mg_1)}\mV^0_\lambda\]
where trivial action $\mg_1\mV^0_\lambda=0$ is assumed. If this module is irreducible, it is equal to the corresponding irreducible module $\mV_\lambda$, the weight $\lambda$ and the module $\mV_\lambda$ are then called typical. Otherwise $K_\lambda$ has a unique maximal submodule and the corresponding quotient is isomorphic to $\mV_\lambda$. Such representations and weights are known as atypical, see \cite{MR0486011}.

\begin{lemma}
\label{KacBGG}
Consider $\mg$ a basic classical Lie superalgebra of type I and a parabolic subalgebra $\Mp$ such that $\ml\subset \mg_0$. For each integral dominant weight $\lambda$ there is a resolution of $K_\lambda$ of the form
\begin{eqnarray*}
&&0\rightarrow\bigoplus_{w\in W^1(\dim\mn_0)}\cU(\mg)\otimes_{\cU(\Mp)}{M(w\cdot\lambda)} \to\cdots\to \bigoplus_{w\in W^1(j)}\cU(\mg)\otimes_{\cU(\Mp)}{M(w\cdot\lambda)}\to\cdots\\
&&\to \bigoplus_{w\in W^1(1)}\cU(\mg)\otimes_{\cU(\Mp)}{M(w\cdot\lambda)}\to \cU(\mg)\otimes_{\cU(\Mp)}{M(\lambda)} \to K_\lambda\to0,
\end{eqnarray*}
with $M(\mu)$ the irreducible $\ml$-module with highest weight $\mu$ and $W^1$ the coset of the Weyl groups corresponding to $\mg_0$, $\Mp_0=\ml+\mn_0$, where $\mn=\mn_0+\mg_1$ and $w\cdot\lambda=w(\lambda+\rho_0)-\rho_0=w(\lambda+\rho)-\rho$ the corresponding action.
\end{lemma}
\begin{proof}
The functor $M\to \cU(\mg)\otimes_{\cU(\mg_0+\mg_1)}M$ from $\cC(\mg_0+\mg_1,\ml)$ to $\cC(\mg,\ml)$ is exact, where $\cC(\ma,\mc)$ stands for the category of $\ma$-modules which are locally $\cU(\mc)$-finite and semisimple as $\mc$-modules, see Lemma 5.2 in \cite{MR2059616}. Applying this exact induction functor on the BGG resolution in \cite{lepowsky} for $\mV_\lambda^0$ then yields the desired resolution.
\end{proof}

\begin{corollary}
Consider $\mg$ a basic classical Lie superalgebra of type I and a parabolic subalgebra $\Mp$ such that $\ml\subset \mg_0$. Each typical irreducible finite dimensional $\mg$-module has a resolution in terms of generalised Verma modules of finite length, given in Lemma \ref{KacBGG}.

For a general Kac module, the property
\[H_k(\overline{\mn},K_\lambda)=\bigoplus_{w\in W^1(k)}\mC_{w\cdot \lambda}\]
holds for $k\le\dim\mn_0$ and $H_k(\overline{\mn},K_\lambda)=0$ for $k>\dim\mn_0$.
\end{corollary}

\begin{remark}
It can be proved that the typical modules for basic classical Lie superalgebras $\mg$ of type I are the only simple modules that possess BGG resolutions (weak or strong) that have finite length. This is a consequence of the fact that only typical modules possess the property that the $\overline{\mn}$-homology groups vanish after a certain degree. This also implies that typical simple $\mg$-modules are the only simple $\mg$-modules that have finite projective dimension as $\overline{\mn}$-modules.
\end{remark}

Finally we state some related results for singly atypical modules of Lie superalgebras of type I. These
are simple modules for which the highest weight $\lambda$ satisfies $\langle \lambda+\rho, \gamma\rangle = 0$ for exactly one positive isotropic
root $\gamma$. Then $\lambda$ and $\mV(\lambda)$ are said to be singly atypical of type $\gamma$. For $\mg = \mathfrak{osp}(2|2n)$ all non-typical simple
modules are singly atypical. It follows from the results in [42, 43] that singly atypical modules possess
BGG resolutions for parabolic subalgebra equal to $\mg_0+\mg_1$. Typical modules are equal to the Kac module
and therefore possess a trivial BGG resolution of this type.

\begin{lemma} 
For $\mg$ a basic classical Lie superalgebra of type I and $\mV(\lambda)$ singly atypical of type $\gamma$, $\mV(\lambda)$ has
an infinite BGG resolution in terms of Kac modules of the form
\[\cdots \to K_{\lambda^{(j)}} \to\cdots\to K_{\lambda^{(1)}} \to K_{\lambda}\to \mV(\lambda) \to 0.\]
Here the weights $\lambda^{(j)}$ can be explicitly obtained by the procedure described in [42, 43]. For the particular case where $\gamma$ is the simple isotropic root, $\lambda^{(j)}=\lambda-j\gamma$ holds.
\end{lemma}
\begin{proof} In [42, 43] the structure of the Kac modules for singly atypical weights was studied, which led to
the conclusion that there is a short exact sequence of the form $\mV(\lambda^{(1)}) \hookrightarrow K_\lambda \tto \mV(\lambda)$, for a certain weight $\lambda^{(1)}$ which is again singly atypical. The BGG resolution
follows immediately from this observation.
\end{proof}
Since all modules for $\mathfrak{osp}(2|2n)$ are typical or singly atypical this result includes the first statement in
Theorem 6.3.

\begin{appendix}
\section{Appendix: The Hopf superalgebra $\cU(\mg)$ and the twisted de Rham operator}

In this appendix we show how the operator $d:\cJ(C^k(\overline{\mn},\mV))\to \cJ(C^{k+1}(\overline{\mn},\mV))$ from Theorem \ref{thmdefop} is obtained from the same operator for the case $\mV=0$ by a $\mg$-module isomorphism between $\cJ(C^k(\overline{\mn},\mV))$ in Definition \ref{defjet} and the tensor product representation $\cJ(\wedge^k{\mn})\otimes\mV$. In the second part of the appendix we prove that the operator $d$ for the case $\mV=0$ can be rewritten as a standard exterior derivative. Therefore we can interpret the operator $d:\cJ(C^k(\overline{\mn},\mV))\to \cJ(C^{k+1}(\overline{\mn},\mV))$ as a twisted de Rham operator. These results are necessary to prove the exactness of the operator $d$, which is equivalent with the exactness of the BGG complex if $\partial$ and $\partial^\ast$ are disjoint, as is proved in Section \ref{BGG}.

The universal enveloping algebra of a Lie superalgebra has the structure of a supercocommutative Hopf superalgebra, see \cite{MR0252485}. The comultiplication $\Delta:\cU(\mg)\to\cU(\mg)\otimes\cU(\mg)$, antipode $S:\cU(\mg)\to\cU(\mg)$, multiplication $m:\cU(\mg)\otimes\cU(\mg)\to\cU(\mg)$ and co-unit $\varepsilon:\cU(\mg)\to\mC$ are generated by
\begin{eqnarray*}
\Delta(A)=A\otimes 1+1\otimes A&&S(A)=-A\\
m(A\otimes B)=AB&&\varepsilon(A)=0
\end{eqnarray*}
for $A,B\in\mg$ and are superalgebra morphisms, where the multiplication on $\cU(\mg)\otimes\cU(\mg)$ is defined as $(u\otimes v)(y\otimes z)=(-1)^{|y||v|}uy\otimes vz$. Basic properties that we will need are
\begin{eqnarray*}
m\circ\left(S\otimes 1\right)\circ\Delta&=&\varepsilon\qquad\mbox{and}\\
(\Delta\otimes 1)\circ\Delta&=&(1\otimes \Delta)\circ\Delta.
\end{eqnarray*}
We use the Sweedler notation $\Delta (U)=\sum_j U^j_1\otimes U^j_2$ for $U\in\cU(\mg)$.

In order to describe the morphism structure between $\cJ(C^k(\overline{\mn},\mV))$ and $\cJ(\wedge^k{\mn})\otimes\mV$ properly we need some extra notations. For every $A\in\mg$ the action of $A^\mV$ on $C^k(\overline{\mn},\mV)$ is given by the tensor product action on $\mV$ and trivial action on $\wedge^k\mn$, so only the $\mZ_2$-gradation of $\mn$ needs to be taken into account. This notation extends to $U^\mV$ for $U\in\cU(\mg)$. Likewise, for $Z\in \Mp$, we define the action $Z^\cJ$ on $C^k(\overline{\mn},\mV)$ as the tensor product action of the adjoint action on $\wedge^k\mn$ and regarding $\mV$ as a sum of trivial representations. In particular $Z^\mV+Z^\cJ$ gives the ordinary representation structure of $\Mp$ on $C^k(\overline{\mn},\mV)$. Note that the $\mg$-module structure of $\cJ(\wedge^k{\mn})\otimes\mV$ corresponds to the tensor product, i.e.
\begin{eqnarray*}
\left(A\beta\right)(U)&=&-(-1)^{|A||U|}\beta(AU)+(-1)^{|A||U|}A^\mV\beta(U)\qquad \mbox{for }A\in\mg,\quad \beta\in\cJ(\wedge^k{\mn})\otimes\mV\mbox{ and }U\in\cU(\mg).
\end{eqnarray*}

We define the $\mg$-module morphism $d^\cJ:\cJ(\wedge^k{\mn})\otimes\mV\to \cJ(\wedge^{k+1}{\mn})\otimes\mV$ to be the operator $d:\cJ(\wedge^k{\mn})\to \cJ(\wedge^{k+1}{\mn})$ as defined in Theorem \ref{thmdefop} in case $\mV=0$ which is extended trivially to $\cJ(\wedge^k{\mn})\otimes\mV$.

\begin{theorem}
\label{twisted}
The $\mg$-module morphism $\chi$ between $\cJ(C^k(\overline{\mn},\mV))$ and $\cJ(\wedge^k{\mn})\otimes\mV$ which sends $\alpha\in\cJ(C^k(\overline{\mn},\mV))$ to $\widetilde{\alpha}\in\cJ(\wedge^k{\mn})\otimes\mV$ defined as
\begin{eqnarray*}
\widetilde{\alpha}(U)=\sum_j \left(U^j_1\right)^\mV\left( \alpha(U^j_2)\right)
\end{eqnarray*}
is an isomorphism. Moreover it satisfies the property $d=\chi^{-1}\circ d^\cJ\circ\chi$ with $d$ given in Theorem \ref{thmdefop} and $d^\cJ$ as defined above.
\end{theorem}
\begin{proof}
The calculation
\begin{eqnarray*}
(A\widetilde{\alpha})(U)&=&(-1)^{|A||U|}A^\mV \widetilde{\alpha}(U)-\sum_j(-1)^{|A||U|}A^\mV \left(U_1^j\right)^\mV{\alpha}(U^j_2)-\sum_j(-1)^{|A||U|+|A||U^j_1|} \left(U_1^j\right)^\mV{\alpha}(AU^j_2)\\
&=&-\sum_j(-1)^{|U^j_2||A|}\left(U^j_1\right)^{\mV}\alpha(A U^j_2)=\widetilde{A\alpha}(U)
\end{eqnarray*}
shows that the linear map $\chi$ is a $\mg$-module morphism. The calculation
\begin{eqnarray*}
\widetilde{\alpha}(UZ)&=&\sum_j(-1)^{|U_2^j||Z|}\left(U_1^j\right)^\mV Z^\mV\alpha(U_2^j)-\sum_j(-1)^{|U_2^j||Z|}\left(U_1^j\right)^\mV Z \left(\alpha(U_2^j)\right)\\
&=&-\sum_j(-1)^{|U||Z|}Z^\cJ\left(U_1^j\right)^\mV  \left(\alpha(U_2^j)\right)=-(-1)^{|U||Z|}Z^\cJ\left(\widetilde{\alpha}(U)\right)
\end{eqnarray*}
for $Z\in\Mp$ shows that the image of $\chi$ is inside $\cJ(\wedge^k{\mn})\otimes\mV$.

The inverse of $\chi$ is defined as $\chi^{-1}(\beta)(U)=\sum_j S(U^j_1)^\mV\beta(U^j_2)$ for $\beta\in\cJ(\wedge^k{\mn})\otimes\mV$ and $U\in\cU(\mg)$. The proof that this is the inverse follows immediately from the relation
\begin{eqnarray*}
(m\otimes 1)\circ(S\otimes \Delta)\circ\Delta=(m\otimes 1)\circ((S\otimes1)\circ \Delta\otimes 1)\circ\Delta=(\varepsilon\otimes 1)\circ\Delta=1.
\end{eqnarray*}

Also the fact that $d=\chi^{-1}\circ d^\cJ\circ\chi$ holds follows from a direct calculation and the relation
\begin{eqnarray*}
\partial f&=&\partial^\cJ f+\xi_a\wedge (\xi_a^\ddagger)^\mV f
\end{eqnarray*}
for $f\in C^k(\overline{\mn},\mV)$ and $\partial^\cJ$ the coboundary operator on $C^k(\overline{\mn},0)=\Lambda^k\mn$ trivially extended to $C^k(\overline{\mn},\mV)$, which follows from the proof of Lemma \ref{propstandder}.
\end{proof}

The vector spaces $\cJ(\Lambda^k\mn)=\mbox{Hom}_{\cU(\Mp)}(\cU(\mg),\Lambda^k\mn)$ are naturally isomorphic to Hom$(\cU(\overline{\mn}),\Lambda^k\mn)$ by the Poincar\'e-Birkhoff-Witt theorem. The operator $d$ from Theorem \ref{thmdefop} for $\mV=0$ remains identically defined under this identification. The space Hom$(\cU(\overline{\mn}),\mC)$ becomes an algebra with multiplication defined by $(\alpha\beta)(U)=\sum_j\alpha(U^j_1)\beta(U^j_2)$ with $\alpha,\beta\in$Hom$(\cU(\overline{\mn}),\mC)$ and $U\in\cU(\overline{\mn})$, where now we consider the Hopf superalgebra $\cU(\overline{\mn})$. This multiplication extends trivially to the case where $\alpha\in$Hom$(\cU(\overline{\mn}),\mC)$ and $\beta\in$Hom$(\cU(\overline{\mn}),\Lambda^k\mn)$. Then the operator $d$ can be rewritten as in the following theorem.
\begin{theorem}
\label{extder}
Consider $d_k:$ {\rm Hom$(\cU(\overline{\mn}),\Lambda^k\mn)\to$ Hom$(\cU(\overline{\mn}),\Lambda^{k+1}\mn)$} induced from the operator $d$ in Theorem \ref{thmdefop} with $\mV=0$. There are elements $\theta_a\in$ {\rm Hom}$(\cU(\overline{\mn}),\mn)$ such that $\theta_a(1)=\xi_a$ for $\{\xi_a\}$ a basis of $\mn$ and 
\begin{eqnarray*}
d_k \circ\left(\theta_a\wedge\right)=\left(\theta_a\wedge\right)\circ d_{k-1}&\mbox{and}& d_0=\sum_a\theta_a\partial_{x_a}
\end{eqnarray*}
hold with $\partial_{x_a}$ supercommuting endomorphisms on {\rm Hom}$(\cU(\overline{\mn}),\mC)$ satisfying the Leibniz rule.
\end{theorem}
It is clear that for each case where the radical $\mn$ is abelian, for example $\mg=\mathfrak{gl}(m|n)$ with $\mg_{\overline{0}}=\mathfrak{gl}(m)\oplus\mathfrak{gl}(n)$ contained in the parabolic subalgebra $\Mp$, this property is immediate. \begin{proof}
As a vector space the isomorphism Hom$(\cU(\overline{\mn}),\mC)\cong S(\mn)$ holds, with $S(\mn)=\oplus_{j=0}^\infty S^j(\mn)$ the supersymmetric tensor powers of $\mn$. It then follows that all endomorphisms on Hom$(\cU(\overline{\mn}),\mC)$ satisfying the Leibniz rule can be written in terms of a commuting basis $\{\partial_{x_b}\}$, in an expansion with Hom$(\cU(\overline{\mn}),\mC)$-valued coefficients. Therefore the operations $\partial_{\xi_a^\ddagger}$ defined as
\begin{eqnarray*}
\left(\partial_{\xi_a^\ddagger}\alpha\right)(U)&=&(-1)^{|U||\xi_a|}\alpha(U\xi_a^\ddagger)
\end{eqnarray*}
can be expanded as $\partial_{\xi_a^\ddagger}=\sum_b f_{ab}\partial_{x_b}$ for $f_{ab}\in$ Hom$(\cU(\overline{\mn}),\mC)$. Then the elements $\theta_b$ of Hom$(\cU(\overline{\mn}),\mn)$ are defined by $\theta_b=\sum_a \xi_a f_{ab}$. By the fact that $d_1\circ d_0=0$, see proof of Theorem \ref{Rhamcomplex}, it follows that $d_1(\theta_a)=0$ since $\theta_a=d_0(x_a)$ with $x_a\in$Hom$(\cU(\overline{\mn}),\mC)$ canonically defined by the operators $\{\partial_{x_b}\}$. Then it follows easily that $d_k(\theta_a\wedge\alpha)=d_1(\theta_a)\wedge\alpha-\theta_a\wedge d_{k-1}(\alpha)$ for $\alpha\in$Hom$(\cU(\overline{\mn}),\Lambda^{k-1}\mn)$, which concludes the proof.
\end{proof} 

\end{appendix}


\end{document}